\documentclass[12pt,oneside,english]{amsart}
\usepackage[T1]{fontenc}
\usepackage[latin9]{inputenc}
\usepackage{color}
\usepackage{fancybox}
\usepackage{calc}
\usepackage{mathrsfs}
\usepackage{amsthm}
\usepackage{amsbsy}
\usepackage{amssymb}
\usepackage{babel}

\usepackage{graphicx}
\usepackage{amsmath}
\usepackage{hyperref}
\usepackage[capitalise]{cleveref}
\crefname{theorem}{Theorem}{Theorems}
\crefname{lemma}{Lemma}{Lemmas}
\crefname{corollary}{Corollary}{Corollaries}
\crefname{proposition}{Proposition}{Propositions}
\crefname{conjecture}{Conjecture}{Conjectures}
\crefname{question}{Question}{Questions}
\crefname{definition}{Definition}{Definitions}
\crefname{example}{Example}{Examples}
\crefname{remark}{Remark}{Remarks}
\crefname{question}{Question}{Questions}
\crefname{enumi}{}{}
\crefname{equation}{}{}

\usepackage[margin=1in]{geometry}

\usepackage[all,cmtip]{xy}

\usepackage{tikz}
\usetikzlibrary{arrows}
\usetikzlibrary{decorations.pathreplacing}

\allowdisplaybreaks

\makeatletter
\numberwithin{equation}{section}
\numberwithin{figure}{section}

  \theoremstyle{plain}
  \newtheorem{theorem}{\protect\theoremname}[section]
  \newtheorem{lemma}[theorem]{\protect\lemmaname}
  \newtheorem{proposition}[theorem]{\protect\propositionname}

	\newtheorem{question}[theorem]{\protect\questionname}
	
	\theoremstyle{definition}
  
	\theoremstyle{remark}
  \newtheorem{remark}[theorem]{\protect\remarkname}

  \theoremstyle{plain}
	\newtheorem*{theorem*}{\protect\theoremname}
  \newtheorem*{lemma*}{\protect\lemmaname}
  \newtheorem*{proposition*}{\protect\propositionname}
  \newtheorem*{corollary*}{\protect\corollaryname}
	\newtheorem{conjecture*}{\protect\conjecturename}
	\newtheorem{question*}{\protect\questionname}
	
	\theoremstyle{definition}
  \newtheorem*{definition*}{\protect\definitionname}
  \theoremstyle{remark}
  \newtheorem*{remark*}{\protect\remarkname}
	
	\providecommand{\theoremname}{Theorem}
  \providecommand{\lemmaname}{Lemma}
  \providecommand{\propositionname}{Proposition}
  \providecommand{\corollaryname}{Corollary}
  \providecommand{\conjecturename}{Conjecture}
	\providecommand{\questionname}{Question}
	\providecommand{\definitionname}{Definition}
  \providecommand{\remarkname}{Remark}

\makeatother

\global\long\def\dee{\mathrm{d}}

\global\long\def\DJ{D_J}
\global\long\def\DJad{D^*_J}
\global\long\def\RJ{R_J}

\DeclareFontFamily{U}{MnSymbolC}{}
\DeclareSymbolFont{MnSyC}{U}{MnSymbolC}{m}{n}
\DeclareFontShape{U}{MnSymbolC}{m}{n}{
    <-6>  MnSymbolC5
   <6-7>  MnSymbolC6
   <7-8>  MnSymbolC7
   <8-9>  MnSymbolC8
   <9-10> MnSymbolC9
  <10-12> MnSymbolC10
  <12->   MnSymbolC12}{}
\DeclareMathSymbol{\intprod}{\mathbin}{MnSyC}{'270}

\global\long\def\cE{\mathcal{E}}

\global\long\def\ka{\mbox{\large{$\kappa$}}}

\global\long\def\oneb{\bar{1}}

\global\long\def\Ab{\bar{A}}
\global\long\def\Bb{\bar{B}}

\global\long\def\thetah{\hat{\theta}}
\global\long\def\nablah{\hat{\nabla}}

\global\long\def\btheta{\boldsymbol{\theta}}
\global\long\def\bh{\boldsymbol{h}}
\global\long\def\Levi{\boldsymbol{L}}

\global\long\def\ups{\upsilon}
\global\long\def\Ups{\Upsilon}

\global\long\def\scrK{\mathscr{K}}

\global\long\def\fg{\mathfrak{g}}

\global\long\def\cT{\mathcal{T}}
\global\long\def\cA{\mathcal{A}}
\global\long\def\cG{\mathcal{G}}

\global\long\def\hook{\lrcorner}

\newcommand{\norm}[1]{\left\lVert#1\right\rVert}

\def\sideremark#1{\ifvmode\leavevmode\fi\vadjust{\vbox to0pt{\vss
 \hbox to 0pt{\hskip\hsize\hskip1em
 \vbox{\hsize3cm\tiny\raggedright\pretolerance10000
  \noindent #1\hfill}\hss}\vbox to8pt{\vfil}\vss}}}%

\begin{document}

\title[]{Bounded strictly pseudoconvex domains in $\mathbb{C}^2$ \\ with obstruction flat boundary}

\author{Sean N.\ Curry}
\author{Peter Ebenfelt}
\thanks{The second author was supported in part by the NSF grant DMS-1600701.}

\begin{abstract}
On a bounded strictly pseudoconvex domain in $\mathbb{C}^n$, $n>1$, the smoothness of the Cheng-Yau solution to Fefferman's complex Monge-Ampere equation up to the boundary is obstructed by a local curvature invariant of the boundary. For bounded strictly pseudoconvex domains in $\mathbb{C}^2$ which are diffeomorphic to the ball, we motivate and consider the problem of determining whether the global vanishing of this obstruction implies biholomorphic equivalence to the unit ball. In particular we observe that, up to biholomorphism, the unit ball in $\mathbb{C}^2$ is rigid with respect to deformations in the class of strictly pseudoconvex domains with obstruction flat boundary. We further show that for more general deformations of the unit ball, the order of vanishing of the obstruction equals the order of vanishing of the CR curvature. Finally, we give a generalization of the recent result of the second author that for an abstract CR manifold with transverse symmetry, obstruction flatness implies local equivalence to the CR $3$-sphere.
\end{abstract}
\subjclass[2010]{Primary 32V15, 32T15; Secondary 32H02, 32W20}

\maketitle

\section{Introduction}

Let $\Omega\subset \mathbb{C}^n$, $n>1$, be a bounded strictly pseudoconvex domain with smooth boundary $\partial \Omega$. It is well known that the domain $\Omega$ is determined up to biholomorphism by the CR geometry of its boundary $\partial \Omega$. There are several interrelated approaches to studying the CR geometry of $\partial \Omega$, and the biholomorphic geometry of $\Omega$. In \cite{Fefferman1976, Fefferman1979} Fefferman proposed the study of these geometries, and in particular of the CR boundary invariants, via the formal asymptotics of the Dirichlet problem
\begin{equation} \label{eqn:FeffermanComplexMongeAmpere}
\left\{ \begin{array}{l}
\mathcal{J}(u):= (-1)^n \,\mathrm{det} \left( \begin{array}{ c c} u & u_{z^{\bar{k}}}\\ u_{z^{j}} & u_{z^{j}z^{\bar{k}}}\end{array} \right) = 1\;\,\mathrm{in}\;\, \Omega,\\
u=0 \;\,\mathrm{on}\;\, \partial \Omega
\end{array}\right.
\end{equation}
with $u>0$ in $\Omega$. Fefferman's equation \cref{eqn:FeffermanComplexMongeAmpere} governs the existence of a complete K\"ahler-Einstein metric on $\Omega$, $-\log(u)$ being the K\"ahler potential. In \cite{ChengYau1980} Cheng and Yau proved the existence of a unique solution $u$ to Fefferman's equation with $u\in C^{\infty}(\Omega)\cap C^{n+\frac{3}{2}-\epsilon}(\overline{\Omega})$, $\epsilon>0$. Subsequently, Lee and Melrose \cite{LeeMelrose1982} showed that the Cheng-Yau solution $u$ has an asymptotic expansion of the form
\begin{equation}\label{eqn:Lee-Melrose-asymptotics}
u \sim \rho \sum_{k=0}^{\infty} \eta_k (\rho^{n+1}\log \rho)^k, \quad \eta_k \in C^{\infty}(\overline{\Omega})
\end{equation}
where $\rho$ is a smooth defining function for $\Omega$ satisfying $\mathcal{J}(\rho)=1+O(\rho^{n+1})$. Such a defining function $\rho$ always exists by \cite{Fefferman1976}, and is unique mod $O(\rho^{n+2})$; one refers to $\rho$ as a \emph{Fefferman defining function}. It follows that the Cheng-Yau solution $u$ is in $C^{n+2-\epsilon}(\overline{\Omega})$, $\epsilon > 0$. While the solution $u$, and hence each $\eta_k$ mod $O(\rho^{\infty})$, is globally uniquely determined, Graham \cite{Graham1987a, Graham1987b} showed that the coefficients $\eta_k$ mod $O(\rho^{n+1})$ are locally uniquely determined by $\partial \Omega$ (and independent of the choice of Fefferman defining function $\rho$). Moreover, he showed that if the coefficient $\eta_1$ of the first log term vanishes on $\partial \Omega$ then $\eta_k$ vanishes to infinite order at the boundary for all $k \geq 1$. Thus $\eta_1|_{\partial\Omega}$ is precisely the obstruction to boundary smoothness of the Cheng-Yau solution to Fefferman's equation. The local invariant $b\eta_1:=\eta_1|_{\partial\Omega}$ of the boundary $\partial \Omega$ is called the \emph{obstruction function}. 
The local invariant $b\eta_1$ of $\partial \Omega$ depends on the embedding in $\mathbb{C}^n$, but transforms as a density under local ambient biholomorphisms and so defines a weighted CR invariant (denoted in the abstract setting by $\mathcal{O}$). In particular, the vanishing of $b\eta_1$ is a CR invariant condition for a strictly pseudoconvex hypersurface $M$ in $\mathbb{C}^n$. If $M$ is a strictly pseudoconvex hypersurface for which the obstruction function vanishes, then $M$ is said to be \emph{obstruction flat}. If $M$ is locally CR equivalent to the unit sphere, then we say that $M$ is \emph{CR flat} (or \emph{CR spherical}). For $\Omega$ the unit ball in $\mathbb{C}^n$ the solution to Fefferman's equation is $u=1-\norm{z}^2$, which is smooth up to the boundary, hence $b\eta_1$ vanishes for the unit sphere $\mathbb{S}^{2n-1}\subset \mathbb{C}^n$, $n>1$. By \cite[Proposition 4.14]{Graham1987a} there are (local) real analytic strictly pseudoconvex hypersurfaces in $\mathbb{C}^n$, $n>1$, not locally CR equivalent to the sphere, for which the local invariant $b\eta_1$ vanishes identically. In this article we consider the problem of determining whether this is possible globally for the boundary of a smooth bounded strictly pseudoconvex domain in $\mathbb{C}^2$. The question of whether global obstruction flatness implies CR flatness is also an interesting problem in higher dimensions, but in the $\mathbb{C}^2$ case a full answer to this question seems at present more attainable. In the $\mathbb{C}^2$ case this problem is also closely connected with a well known conjecture concerning the weak singularity in the asymptotic expansion of the Bergman kernel (see \cref{sec:BergmanKernel}). A strong form of this conjecture asserts that every topologically trivial smooth bounded strictly pseudoconvex domain in $\mathbb{C}^2$ with obstruction flat boundary is biholomorphic to the unit ball. Here we give some preliminary results in this direction.

Our first main result is the following observation:
\begin{theorem}\label{thm:DeformationRigidity}
Let $\Omega_t$, $t\in [0,1]$, be a smooth family of smooth bounded strictly pseudoconvex domains in $\mathbb{C}^2$, with $\Omega_0$ the unit ball. If $\partial\Omega_t$ is obstruction flat for all $t$, then each $\Omega_t$ is biholomorphic to the unit ball $\Omega_0$.
\end{theorem}
\begin{remark*}
Moreover, using the slice theorem of Ch\^eng and Lee \cite{ChengLee1995}, one can show that in \cref{thm:DeformationRigidity} there exists in fact a smooth family of biholomorphisms $\Phi_t:\Omega_t\to\Omega_0$ for $t\in [0,1]$; see the remark following the proof of \cref{thm:DeformationRigidity}.
\end{remark*}
\cref{thm:DeformationRigidity} follows straightforwardly from the work of Ch\^eng and Lee on the Burns-Epstein invariant \cite{ChengLee1990, ChengLee1995}, using a recent observation of Hirachi \cite{Hirachi2014}. The Burns-Epstein invariant is discussed in \cref{sec:Burns-Epstein}, where \cref{thm:DeformationRigidity} is proved. The necessary background on pseudohermitian and CR invariants and on deformations of strictly pseudoconvex hypersurfaces in $\mathbb{C}^2$ are discussed in Sections \ref{sec:pseudohermitian} and \ref{sec:abstract-and-embedded-deformations} respectively.

A more direct approach to studying this problem is to analyze the variational properties of the CR curvature, namely Cartan's umbilicity tensor $Q$ \cite{Cartan1932a, Cartan1932b}, and of the obstruction density $\mathcal{O}$ under abstract and embedded deformations. At the linearized level, on the CR $3$-sphere $(M=S^3, H, J)$ there is a deformation complex
\begin{equation}\label{eqn:DeformationComplex}
0\longrightarrow TM/H \longrightarrow \mathcal{D}ef(M) \longrightarrow \mathcal{C}urv(M) \longrightarrow \mathcal{B}ian(M) \longrightarrow 0,
\end{equation}
governing abstract deformations of the CR structure for which the contact distribution $H$ is held fixed (there is no loss of generality in doing this by a theorem of Gray \cite{Gray1959}). Generically these abstract deformations will not be embeddable \cite{Nirenberg1974, JacobowitzTreves1982}. Given a CR hypersurface $M$ in a complex manifold $\Sigma$, abstract deformations of the induced CR structure $(M, H, J)$ which arise from a $1$-parameter family of strictly pseudoconvex embeddings $\psi_t : M \to \Sigma$ with $\psi_0 = \mathrm{id}$ (i.e. from a \emph{Kuranishi wiggle}) are referred to as \emph{stably embeddable deformations}. It turns out that by complexifying and type decomposing the complex \cref{eqn:DeformationComplex} on the CR $3$-sphere one obtains a bigraded complex in which the linearized operators governing stably embeddable deformations and the CR obstruction density appear. Working with this complex, and applying a sequence of normalizations of the parametrized deformation $\psi_t : S^3 \to \mathbb{C}^2$ we are able to prove:
\begin{theorem}\label{thm:FormalDeformations}
Let $(S^3,H,J_0)$ be the CR $3$-sphere and let $(S^3,H,J_t)$, $t\in[0,\epsilon)$, be a smooth family of stably embeddable deformations. If the CR obstruction density $\mathcal{O}_t$ of $(S^3,H,J_t)$ vanishes to order $k$ at $t=0$ then so does the CR curvature tensor $Q_t$.
\end{theorem}
\cref{thm:FormalDeformations} is in some sense a refinement of \cref{thm:DeformationRigidity}, though it only implies \cref{thm:DeformationRigidity} in the case of real analytic dependence on the deformation parameter $t$. We include this because it may be proved by more elementary and direct methods, and is also of independent interest. One consequence of \cref{thm:FormalDeformations} is that the real ellipsoids close to the sphere cannot be obstruction flat, and hence do not give rise to non-spherical examples of compact obstruction flat hypersurfaces (see \cref{sec:real-ellipsoids}). The deformation complex \cref{eqn:DeformationComplex} is discussed in detail in \cref{sec:DeformationComplex}, where \cref{thm:FormalDeformations} is proved. It should be mentioned here that the corresponding result fails to hold if the abstract deformation is not required to be stably embeddable \cite{CE2018-stably-embeddable-deformations}.

In recent work on this problem the second author \cite{Ebenfelt-arxiv2016} has shown that for compact abstract CR $3$-manifolds with transverse symmetry (which implies local embeddability) obstruction flatness implies CR flatness. Using a new approach, we extend this result by relaxing the transversality condition on the infinitesimal CR symmetry. We prove:
\begin{theorem}\label{thm:Symmetries}
Let $(M,H,J)$ be a compact CR $3$-manifold with infinitesimal CR symmetry. If $(M,H,J)$ is obstruction flat, then the CR structure is locally spherical.
\end{theorem}
\cref{thm:Symmetries} covers a broad class of domains in $\mathbb{C}^2$ not covered by the main theorem in \cite{Ebenfelt-arxiv2016}. For example, if $\Omega\subset \mathbb{C}^2$ is the strictly pseudoconvex domain given by
\begin{equation*}
|w|^2 + f(z,\bar{z}) < c,
 \end{equation*}
 where $f$ is a proper strictly subharmonic function and $c>\mathrm{min}f$ is a constant, then $X=\mathrm{Re}\,(iw\frac{\partial}{\partial w})$ is an infinitesimal CR symmetry of $M=\partial \Omega$, which is transverse to the CR contact distribution on $M$ except along the curve $M\cap\{w=0\}$. Moreover, it is highly unclear how one could obtain the result in \cref{thm:Symmetries} by modifying the approach of \cite{Ebenfelt-arxiv2016}, since the latter relies on being able to work with the pseudohermitian calculus associated with a global contact form $\theta$ for $H$ normalized by $\theta(X) = 1$. We instead develop a new approach to the problem based on the CR invariant calculus associated with the CR Cartan connection. This makes essential use of the work of \v{C}ap \cite{Cap2008} on infinitesimal symmetries and deformations of parabolic geometries (of which hypersurface type CR geometries are an example). These ideas are developed in \cref{sec:TractorCalculus}, where \cref{thm:Symmetries} is proved. We believe that this approach will be highly useful in further work on this problem.

\begin{remark}
It is worthwhile here to point out an analogy with conformal geometry in four dimensions. It is well known that a CR $3$-manifold is obstruction flat if and only if its Fefferman space (a conformal Lorentzian $4$-manifold) is Bach flat. While it is easy to produce examples of compact, Bach flat $4$-manifolds, e.g., any (conformally) Einstein $4$-manifold is Bach flat, it is also known that a $4$-dimensional Fefferman space which is locally conformally Einstein is necessarily locally conformally flat (meaning that the underlying CR structure is locally spherical). Thus, this observation by itself does not provide examples of compact, obstruction flat CR $3$-manifolds that are not locally spherical. In fact, the authors are not aware of any such examples.

The analogy with conformal geometry is useful in the proof of \cref{thm:Symmetries}. In particular, in $4$-dimensional conformal geometry the Bach tensor can be interpreted as the Yang-Mills current for the Cartan/tractor curvature (see, e.g., \cite{BransonGover2001, NurowskiSparling2003}). In our proof of \cref{thm:Symmetries}, we make use of an analogous interpretation for the obstruction density of a CR $3$-manifold (see \cref{lem:obstruction-flat-as-divergence}).
\end{remark}

\subsection*{Acknowledgements} The authors would like to thank Mike Eastwood, Rod Gover, Robin Graham, Kengo Hirachi, Bernhard Lamel, Jack Lee, Pawel Nurowski and Paul Yang for helpful conversations. Part of this work was carried out while the first author was visiting the Banach Centre at IMPAN in Warsaw for the Simons Semester `Symmetry and Geometric Stuctures' (BCSim-2017-s06).

\section{The log term in the asymptotic expansion of the Bergman kernel}\label{sec:BergmanKernel}

Let $\Omega$ be a domain in $\mathbb{C}^n$. The Bergman kernel of $\Omega$ is the integral kernel $K(z,w)$ for the orthogonal projection operator from $L^2(\Omega)$ to the Hardy space $\mathcal{A}^2(\Omega)$ of $L^2$ holomorphic functions on $\Omega$. Given any orthonormal basis $\{h_j\}_{j=1}^{\infty}$ of $\mathcal{A}^2(\Omega)$ the Bergman kernel may be written as $K(z,w)=\sum_{j=1}^{\infty}h_j(z)\overline{h_j(w)}$. When $\Omega$ is the unit ball in $\mathbb{C}^n$ the Bergman kernel is given by $K(z,w)=\frac{n!}{\pi^n}(1-z\cdot \bar{w})^{-(n+1)}$. For $\Omega$ a smooth strictly pseudoconvex domain in $\mathbb{C}^n$ it was shown in \cite{Fefferman1974} that the Bergman kernel along the diagonal may be written as
\begin{equation}
K = \frac{\phi}{\rho^{n+1}} + \psi \log \rho
\end{equation}
where $\rho$ is a defining function for the boundary, and $\phi,\psi\in C^{\infty}(\overline{\Omega})$. Moreover, taking $\rho$ to be a Fefferman defining function, $\phi$ mod $O(\rho^{n+1})$ and $\psi$ mod $O(\rho^{\infty})$ are locally uniquely determined by $\partial \Omega$ (and independent of the choice of Fefferman defining function $\rho$). For the unit ball one may take $\rho = u = 1-\norm{z}^2$, so that $\phi=\frac{n!}{\pi^n}$ and $\psi=0$. A problem posed by many is that of classifying those strictly pseudoconvex domains for which the so-called `weak singularity' $\psi$ mod $O(\rho^{\infty})$ in the asymptotic expansion of the Bergman kernel vanishes. In \cite{Ramadanov1981}, Ramadanov conjectured that if $\psi$ vanishes to infinte order on the boundary of $\Omega$ then $\Omega$ must be biholomorphic to the unit ball. In $\mathbb{C}^2$ a local version of this conjecture holds by work of Graham (who attributes the result to Burns). In \cite{Graham1987b} (cf. \cite{BoutetdeMonvel1988}) Graham expanded $\psi$ in powers of a Fefferman defining function $\rho$, in the $\mathbb{C}^2$ case, to obtain
\begin{equation}\label{eqn:log-term-expansion}
\psi = -\frac{6}{\pi^2}\eta_1 + k|Q|^2\rho + O(\rho^2)
\end{equation}
where $\eta_1$ is as in \cref{eqn:Lee-Melrose-asymptotics}, $Q$ is the Cartan umbilicity tensor of the boundary, and $k$ is a nonzero constant (explicitly computed in \cite{HirachiKomatsuNakazawa1993}).
Using that $\eta_1|_{\partial\Omega}=0$ implies $\eta_1 = O(\rho^\infty)$, it follows from \cref{eqn:log-term-expansion} that if $\psi=O(\rho^2)$ then $Q$ must vanish identically on $\partial \Omega$ (i.e. $\partial \Omega$ must be locally CR spherical); the argument here is local, so that if one only has $\psi=O(\rho^2)$ in the neighborhood of some point in the boundary, then one may still conclude that the boundary is locally CR spherical in that neighborhood. If the domain $\Omega\subset\mathbb{C}^2$ is taken to be simply connected with connected boundary then one may apply the Riemann mapping theorem of \cite{ChernJi1996} to obtain the result that the vanishing of $\psi$ (to second order on the boundary) implies $\Omega$ is biholomorphic to the unit ball. Such topological assumptions are in fact necessary for biholomorphic equivalence to the unit ball to hold, due to examples of bounded strictly pseudoconvex domains, not biholomorphic to the unit ball but with locally spherical boundary, constructed by Burns and Shnider \cite{BurnsShnider1976}. So the conjecture is resolved in the case of $\mathbb{C}^2$. In higher dimensions the conjecture remains open (though see \cref{rem:Hirachi-IMPAN-Nov2017}); there are some negative results for a natural generalization of this conjecture to complex manifolds \cite{EnglisZhang2010, LoiMossaZuddas2017}, highlighting the global nature of this problem.

Closely related to Ramadanov's conjecture is the following question:

\begin{question}\label{question:log-term-coefficient-on-boundary}
Let $\Omega\subset \mathbb{C}^n$, $n>1$, be a smooth bounded strictly pseudoconvex domain with $\psi|_{\partial\Omega}$ identically zero. Does this imply $\partial\Omega$ is locally CR spherical?
\end{question}
In constrast to the above local resolution for the $\mathbb{C}^2$ case of Ramadanov's conjecture, in any dimension the answer to the local version of \cref{question:log-term-coefficient-on-boundary} is no; this is a global problem. In the $\mathbb{C}^2$ case, by \cref{eqn:log-term-expansion}, \cref{question:log-term-coefficient-on-boundary} is equivalent to the question of whether (global) obstruction flatness of the boundary implies local CR flatness.  In this case the question has been taken up already, e.g., in \cite{BoichuCoeure1983, Nakazawa1994} where bounded Reinhardt domains are considered, for which the answer to the question is positive, and in \cite{Ebenfelt-arxiv2016} where compact CR $3$-manifolds with transverse symmetry are considered, for which obstruction flatness is shown to imply local CR flatness. Our goal in what follows is to prove some further results in this direction.
\begin{remark}
By work of Boutet de Monvel and Sj\"ostrand \cite{BoutetdeMonvelSjostrand1976}, for strictly pseudoconvex domains the Szeg\H{o} kernel enjoys an asymptotic expansion similar to that of the Bergman kernel. Questions analogous to those stated above for the Bergman kernel have also been posed for the Szeg\H{o} kernel, taken with respect to a suitably chosen (CR invariant) surface measure on the boundary \cite{HirachiKomatsuNakazawa1993}. In the $\mathbb{C}^2$ case the coefficient of the log term in the Szeg\H{o} kernel has an expansion similar to the expansion \cref{eqn:log-term-expansion} \cite{HirachiKomatsuNakazawa1993}, which again leads naturally to the question of whether global obstruction flatness implies local CR flatness (see also the discussion in \cite{Ebenfelt-arxiv2016}).
\end{remark}
\begin{remark}\label{rem:Hirachi-IMPAN-Nov2017}
Recently Kengo Hirachi has announced a positive answer to \cref{question:log-term-coefficient-on-boundary} for domains in $\mathbb{C}^n$, $n\geq 3$, whose boundaries are sufficiently near the unit sphere (in his talk at the conference on `Symmetry and Geometric Structures' at IMPAN, Warsaw, November 12-18, 2017). In particular, this implies that the conjecture of Ramadanov is true for small perturbations of the unit ball.
\end{remark}

\section{Pseudohermitian calculus}\label{sec:pseudohermitian}

In this section we recall some standard background material on pseudohermitian and CR structures, and the associated Tanaka-Webster calculus.

Let $M$ be a smooth oriented $3$-manifold. A \emph{contact structure} on $M$ is a rank $2$ subbundle $H\subset TM$ which is nondegenerate in the sense that if $H$ is locally given as the kernel of some $1$-form $\theta$, then $\theta\wedge\dee\theta$ is nowhere vanishing. A \emph{CR structure} on $(M,H)$ is given by a smooth endomorphism $J:H\to H$ such that $J^2=-\mathrm{id}$. We refer to $(M,H,J)$ as a \emph{CR $3$-manifold}. The partial complex structure $J$ on $H\subset TM$ defines an orientation of $H$, and therefore defines an orientation on the annihilator subbundle $H^{\perp}:=\mathrm{Ann}(H)\subset T^*M$. A nowhere vanishing section $\theta$ of $H^{\perp}$ is called a \emph{contact form} for $H$. A contact form $\theta$ is positively oriented if $\dee \theta|_{H}$ is compatible with the orientation of $H$, equivalently, if $\dee \theta( \,\cdot\, , J\,\cdot\,)$ is positive definite on $H$. A CR structure $(M,H,J)$ together with a choice of positively oriented contact form $\theta$ is referred to as a \emph{pseudohermitian structure}. The \emph{Reeb vector field} of a contact form $\theta$ is is the vector field $T$ uniquely determined by $\theta(T)=1$ and $T\intprod \dee\theta =0$.

Given a CR manifold $(M,H,J)$ we decompose the complexified contact distribution $\mathbb{C}\otimes H$
as $T^{1,0}\oplus T^{0,1}$, where $J$ acts by $i$ on $T^{1,0}$ and by $-i$ on $T^{0,1}=\overline{T^{1,0}}$. Let $\theta$ be an oriented contact form on $M$. Let $Z_1$ be a local frame for the \emph{holomorphic tangent bundle} $T^{1,0}$ and $Z_{\oneb}=\overline{Z_1}$, so that $\{T,Z_1,Z_{\oneb}\}$ is a local frame for $\mathbb{C}\otimes TM$. Then the dual frame $\{\theta,\theta^1,\theta^{\oneb}\}$ is referred to as an \emph{admissible coframe} and one has
\begin{equation}\label{eqn:h11-definition}
\dee \theta = i h_{1\oneb} \theta^1 \wedge \theta^{\oneb}
\end{equation}
for some positive smooth function $h_{1\oneb}$. The function $h_{1\oneb}$ is the component of the Levi form $\mathrm{L}_{\theta}(U,\overline{V})=-2i\dee\theta(U,\overline{V})$ on $T^{1,0}$, that is
\begin{equation*}
\mathrm{L}_{\theta}(U^1Z_1,V^{\oneb}Z_{\oneb}) = h_{1\oneb}U^1V^{\oneb}.
\end{equation*}
It is sometimes convenient to scale $Z_1$ so that $h_{1\oneb}=1$, but we will not assume this unless otherwise specified. We write $h^{1\oneb}$ for the multiplicative inverse of $h_{1\oneb}$. The Tanaka-Webster connection associated to $\theta$ is given in terms of such a local frame $\{T,Z_1,Z_{\oneb}\}$ by
\begin{equation*}
\nabla Z_1 = \omega_1{}^{1}\otimes Z_1, \quad \nabla Z_{\oneb} = \omega_{\oneb}{}^{\oneb}\otimes Z_{\oneb}, \quad \nabla T =0
\end{equation*}
where the connection $1$-forms $\omega_1{}^{1}$ and $\omega_{\oneb}{}^{\oneb}$ satisfy
\begin{equation}\label{eqn:pseudohermitian-connection1}
\dee \theta^1 = \theta^1\wedge \omega_1{}^{1} + A^1{}_{\oneb}\,\theta\wedge\theta^{\oneb}, \text{ and}
\end{equation}
\begin{equation}\label{eqn:pseudohermitian-connection2}
\omega_1{}^{1} + \omega_{\oneb}{}^{\oneb} =h^{1\oneb}\dee h_{1\oneb},
\end{equation}
for some function $A^1{}_{\oneb}$. The uniquely determined function $A^1{}_{\oneb}$ is known as the \emph{pseudohermitian torsion}. Components of covariant derivatives will be denoted by adding $\nabla$ with an appropriate subscript, so, e.g., if $u$ is a function then $\nabla_1 u = Z_1 u$ and $\nabla_0\nabla_1 u = T Z_1 u - \omega_1{}^{1}(T)Z_1u$. We may also use $h_{1\oneb}$ and $h^{1\oneb}$ to raise and lower indices, so that $A_{\oneb\oneb}= h_{1\oneb}A^1{}_{\oneb}$ and $A_{11}=h_{1\oneb}A^{\oneb}{}_{1}$, with $A^{\oneb}{}_{1} =\overline{A^1{}_{\oneb}}$.

The \emph{pseudohermitian (scalar) curvature} $R$ is defined by the structure equation
\begin{equation*}
\dee \omega_1{}^{1} = Rh_{1\oneb} \theta^1\wedge \theta^{\oneb} + (\nabla^1 A_{11})\,\theta^1\wedge\theta - (\nabla^{\oneb} A_{\oneb\oneb})\,\theta^{\oneb}\wedge\theta.
\end{equation*}
The torsion of the Tanaka-Webster connection (as an affine connection) is captured by the following formulae, for a smooth function $f$,
\begin{equation*}
\nabla_1\nabla_{\oneb}f-\nabla_{\oneb}\nabla_{1}f  =-i h_{1\oneb}\nabla_{0}f,\quad \text{ and }\quad  \nabla_{1}\nabla_{0}f-\nabla_{0}\nabla_{1}f  =A^{\oneb}{_{1}}\nabla_{\oneb}f.
\end{equation*}
The pseudohermitian curvature $R$ may therefore equivalently be defined by the Ricci identity
\begin{equation}\label{eqn:Ricci-identity}
\nabla_1\nabla_{\oneb} V^1 -\nabla_{\oneb}\nabla_1 V^1 + i h_{1\oneb}\nabla_0 V^1 = Rh_{1\oneb} V^1
\end{equation}
for any local section $V^1Z_1$ of $T^{1,0}$. Commuting $0$ and $1$ (or $\oneb$) derivatives on $V^1Z_1$ gives torsion according to the following formulae
\begin{equation}\label{eqn:Tanaka-Webster-commuting-10-derivatives}
\nabla_1\nabla_{0} V^1 -\nabla_{0}\nabla_1 V^1 - A^{\oneb}{_{1}}\nabla_{\oneb}V^1 = (\nabla^1 A_{11}) V^1, \;\;\text{and}
\end{equation}
\begin{equation*}
\nabla_{\oneb}\nabla_{0} V^1 -\nabla_{0}\nabla_{\oneb} V^1 - A^{1}{_{\oneb}}\nabla_{1}V^1 = (\nabla^{\oneb} A_{\oneb\oneb}) V^1.
\end{equation*}
In dimension $3$, the Bianchi identities of \cite[Lemma 2.2]{Lee1988} reduce to
\begin{equation}\label{eqn:pseudohermitian-Bianchi}
\nabla_0 R = 2\mathrm{Re}\,(\nabla^1 \nabla^1 A_{11}).
\end{equation}

The local calculus on CR manifolds associated with the CR Cartan connection is discussed in more detail in \cref{sec:TractorCalculus}. For now it suffices to recall some basic definitions and formulae in terms of pseudohermitian calculus. The \emph{Cartan umbilical tensor} $Q$ of $(M,H,J)$ is a (weighted) CR invariant, whose vanishing is necessary and sufficient for $(M,H,J)$ to be locally equivalent to the induced CR structure on the unit sphere in $\mathbb{C}^2$. As in \cite{ChengLee1990}, given a choice of contact form $\theta$ we interpret the umbilical tensor $Q$ as an endomorphism of $H$, written locally as
\begin{equation}\label{eqn:Q-endomorphism}
Q=iQ_1{}^{\oneb}\theta^1\otimes Z_{\oneb} - iQ_{\oneb}{}^1\theta^{\oneb}\otimes Z_1.
\end{equation}
By \cite[Lemma 2.2]{ChengLee1990} the component $Q_{11}$ of Cartan's tensor is given by
\begin{equation}\label{eqn:Cartan-umbilical-tensor}
Q_{11} =-\frac{1}{6}\nabla_1\nabla_1 R - \frac{i}{2}RA_{11} + \nabla_0 A_{11} +\frac{2i}{3} \nabla_1\nabla^1 A_{11},
\end{equation}
where we have taken the opposite sign convention. If $\thetah=e^{\Ups}\theta$ is another contact form, then $\hat{Q}=e^{-2\Ups}Q$, so that $Q$ may be thought of more invariantly as a weighted section of $\mathrm{End}(H)$. More precisely, $Q$ may be thought of as a CR invariant section of $\mathrm{End}(H)\otimes (TM/H)^{-2}$, the dependency on the contact form $\theta$ only being introduced when we use $\theta$ to trivialize $TM/H$. The Bianchi identity for the curvature of the CR Cartan connection (see \cref{sec:TractorCalculus}) is equivalent to the following Bianchi identity for $Q$, expressed locally as
\begin{equation}\label{eqn:Q11-Bianchi}
\mathrm{Im}(\nabla^1\nabla^1Q_{11}-iA^{11}Q_{11})=0,
\end{equation}
which may also be seen as a direct consequence of \cref{eqn:pseudohermitian-Bianchi}. The \emph{CR obstruction density} is given locally by
\begin{equation}\label{eqn:obstruction-density}
\mathcal{O}=\frac{1}{3}(\nabla^1\nabla^1Q_{11}-iA^{11}Q_{11}).
\end{equation}
The CR obstruction density $\mathcal{O}$ is again a (weighted) CR invariant. If $\thetah=e^{\Ups}\theta$ is another contact form, then $\hat{\mathcal{O}}=e^{-3\Ups}\mathcal{O}$, so that $\mathcal{O}$ defines a CR invariant section of $(TM/H)^{-3}$. Our convention here has been chosen so that, for a strictly pseudoconvex domain $\Omega\subset\mathbb{C}^2$ we have $b\eta_1=\frac{1}{4}\mathcal{O}$, consistent with \cite{HirachiMarugameMatsumoto2017}. Here $b\eta_1$ is also thought of as a density; to obtain the function which arises as the boundary restriction of $\eta_1$ in the expansion \cref{eqn:Lee-Melrose-asymptotics} one should compute $b\eta_1$ with respect to the contact form $\theta = \mathrm{Re}(i\partial\rho)|_{TM}$ induced by a Fefferman defining function $\rho$ for $\Omega$. Since we are only concerned with obstruction flatness, we will allow ourselves to compute with respect to any contact form and work with the CR obstruction density $\mathcal{O}$.

\begin{remark}\label{remark:obstruction-formula}
The CR invariance of the right hand side of \cref{eqn:obstruction-density} will be made clear in \cref{sec:TractorCalculus}. While it is well known by weight considerations (\cite{Graham1987b}) that one therefore has $\mathcal{O}=c(\nabla^1\nabla^1Q_{11}-iA^{11}Q_{11})$ for some nonzero real constant $c$, and there are various ways to determine the constant $c$ by combining references from the literature, it is hard to find a single self-contained reference for the formula \cref{eqn:obstruction-density}. Here we outline a method for deriving this formula. For the computation of general formulae for local CR invariants there is no loss of generality in restricting to the real analytic case. One may therefore compute the formula for $\mathcal{O}$ by considering a real hypersurface $M$ in $\mathbb{C}^2$, taken to be in Chern-Moser normal form \cite{ChernMoser1974}. Letting $(z,w)$ be coordinates for $\mathbb{C}^2$ one takes the normalized defining function
\begin{equation*}
\rho = 2\mathrm{Im}\,w - |z|^2 - \sum_{k,l\geq 2, j\geq 0}\sum A^j_{k\bar{l}}z^j \bar{z}^l(\mathrm{Re}\,w)^j,
\end{equation*}
 with $A^j_{2\bar{2}}=
A^j_{2\bar{3}}=A^j_{3\bar{3}}=0$ for all $j$, and defines the contact form $\theta = \mathrm{Re}(i\partial\rho)|_{TM}$. Taking $\theta^1=\dee z$ one may then solve \cref{eqn:h11-definition}, \cref{eqn:pseudohermitian-connection1} and \cref{eqn:pseudohermitian-connection2} for $h_{11}$, $\omega_1{}^1$ and $A^1{}_{\oneb}$. It is then a straightforward but tedious exercise to confirm that $\nabla^1\nabla^1Q_{11}-iA^{11}Q_{11} = \frac{1}{48}A^0_{4\bar{4}}$ at the origin. By \cite[Proposition 2.2]{Graham1987b} $b\eta_1=4A^0_{4\bar{4}}$, and since we have taken $\mathcal{O}:=4b\eta_1$ this gives \cref{eqn:obstruction-density}. We shall later also see that the constant in \cref{eqn:obstruction-density} makes results of \cite{ChengLee1990} consistent with \cite{HirachiMarugameMatsumoto2017}.
\end{remark}

The bundle $TM/H$ will play an important role in what follows, and should be thought of as a fundamental \emph{density} bundle on $(M,H)$. Let $\mathcal{M}\subset H^{\perp}$ be the bundle of oriented contact forms, thought of as an $\mathbb{R}_+$ bundle over $M$ in the obvious way, and let $\Theta$ be the tautological $1$-form on $\mathcal{M}$ defined by $\Theta_{\theta}=\theta\circ \pi_*$ where $\pi:\mathcal{M}\to M$ is the natural projection. Then $(\mathcal{M}, \dee \Theta)$ is a symplectic manifold, called the \emph{symplectization} of $(M,H)$. Sections of $TM/H$ may be identified with functions which are homogeneous of degree $1$ on $\mathcal{M}$, and sections of $(TM/H)^w$ with functions homogeneous of degree $w$. Consistent with \cref{sec:TractorCalculus}, we introduce the notation $\cE_{\mathbb{R}}(w,w)$ for $(TM/H)^w$, and $\cE(w,w)$ for the corresponding complex line bundle $\mathbb{C}\otimes\cE_{\mathbb{R}}(w,w)$. (For the case of $\cE_{\mathbb{R}}(1,1)$ we will often still write $TM/H$.) The CR obstruction density $\mathcal{O}$ is an invariant section of $\cE_{\mathbb{R}}(-3,-3)$; we say that $\mathcal{O}$ is a CR density of weight $(-3,-3)$. The term `density' is further justified by the observation that $\cE_{\mathbb{R}}(-2,-2)$ may be canonically identified with the bundle $\Lambda^3$ of top-forms on the oriented manifold $M$. To see this, note that the bundle $\cE_{\mathbb{R}}(-1,-1)=(TM/H)^*$ may be naturally identified with $H^{\perp}$, so that $\cE_{\mathbb{R}}(-2,-2)$ may be identified with $H^{\perp}\otimes H^{\perp}$. The canonical identification of $\cE_{\mathbb{R}}(-2,-2)$ with $\Lambda^3$ is then given by the map $H^{\perp}\otimes H^{\perp}\ni \theta\otimes\theta \mapsto \theta\wedge\dee \theta \in \Lambda^3$. We write $\btheta$ for the tautological section of $T^*M\otimes \cE_{\mathbb{R}}(1,1)$ given by the map $TM\to TM/H=\cE_{\mathbb{R}}(1,1)$.

\section{Infinitesimal symmetries and deformations of CR structures}\label{sec:abstract-and-embedded-deformations}

Here we collect some basic results on infinitesimal symmetries and abstract deformations of CR $3$-manifolds, and on infinitesimal deformations of strictly pseudoconvex hypersurfaces in complex surfaces. The relation between abstract and embedded deformations of CR $3$-manifolds has been much studied, particularly in connection with the realizability problem for abstract CR $3$-manifolds \cite{Bland1994, BlandEpstein1996, BlandDuchamp2011, BurnsEpstein1990b, CaseChanilloYang2016, Epstein1992, Epstein1998, EpsteinHenkin2000, Lempert1992, Lempert1994}. The results we present are well known. See, e.g., \cite{ChengLee1990} for an excellent reference on infinitesimal abstract deformations of CR $3$-manifolds. Our approach to infinitesimal deformations of strictly pseudoconvex hypersurfaces in complex surfaces is based on \cite{BlandEpstein1996}. 

\subsection{Contact Hamiltonian vector fields and infinitesimal CR symmetries}

It is well known that the space $\Gamma(TM/H)$ parametrizes the infinitesimal contact diffeomorphisms of $(M,H)$. Given a section $f$ of $TM/H$ there is a vector field $V_f$ on $M$ uniquely determined by the conditions that $V_f\;\mathrm{mod}\; H = f$ and that the Lie derivative $\mathcal{L}_{V_f}$ preserves $\Gamma(H)$, i.e. that $V_f$ be an infinitesimal contact diffeomorphism. The vector field $V_f$ is referred to as the \emph{contact Hamiltonian vector field} with potential $f$. Often we will fix a background contact form for $H$, and thereby think of $f$ as a smooth function on $M$. By Cartan's formula for $\mathcal{L}_{V_f}\theta$ one then has $V_f = fT + H_f$ where $H_f\in \Gamma(H)$ is determined by $H_f\,\hook\,\dee\theta \equiv -\dee f \;\;(\mathrm{mod}\;\theta)$. Moreover, on a CR manifold $(M,H,J)$ we then have the following local formula:
\begin{lemma}
Let $\theta$ be a contact form for $H$ and $Z_1$ a local frame for $T^{1,0}$. The contact Hamiltonian vector field with potential $f$ is given locally by
\begin{equation}
V_f = fT + i f^1 Z_1 - i f^{\oneb}Z_{\oneb}
\end{equation}
where $f^1=\nabla^1f$ and $f^{\oneb}=\nabla^{\oneb}f$.
\end{lemma}
\begin{proof}
It suffices to check that $H_f:=i f^1 Z_1 - i f^{\oneb}Z_{\oneb}$ satisfies $H_f\,\hook\,\dee\theta \equiv -\dee f \;\;(\mathrm{mod}\;\theta)$. By \cref{eqn:h11-definition} we have
$
H_f \,\hook\,\dee\theta = -f_{\oneb}\theta^{\oneb} - f_1\theta^1 = -\dee f\;\; \mathrm{mod}\; \theta,
$
as required.
\end{proof}
An \emph{infinitesimal CR symmetry} of $(M,H,J)$ is a vector field $V$ whose flow consists of (local) CR diffeomorphisms of $M$. In particular, such a $V$ must be an infinitesimal contact diffeomorphism. An infinitesimal contact diffeomorphism $V=V_f$ is a CR symmetry if and only if $\mathcal{L}_V J =0$ (this being defined since the flow of $V$ preserves $H$). With this in mind we recall:
\begin{lemma}[\cite{ChengLee1990}]
Let $\theta$ be a contact form for $H$ and $Z_1$ a local frame for $T^{1,0}$. If $V=V_f$ is a contact Hamiltonian vector field, then the Lie derivative $\mathcal{L}_V J\in \Gamma(\mathrm{End}(H))$ is given locally by
\begin{equation*}
\mathcal{L}_V J = -2(\nabla_1\nabla^{\oneb}f+iA_1{}^{\oneb}f)\theta^1\otimes Z_{\oneb} - 2(\nabla_{\oneb}\nabla^{1}f-iA_{\oneb}{}^{1}f)\theta^{\oneb}\otimes Z_{1}.
\end{equation*}
\end{lemma}
Following \cite{ChengLee1995} we define a CR invariant second order operator $\DJ: TM/H \to \mathrm{End}(H)$ given by $\DJ f = -\frac{1}{2}\mathcal{L}_{V_f} J$. Choosing a contact form $\theta$, by which we identify $f$ with a smooth function on $M$, and a local frame $Z_1$ for $T^{1,0}$, we have
\begin{equation}\label{eqn:DJ-formula}
\DJ f = (\nabla_1\nabla^{\oneb}f+iA_1{}^{\oneb}f) \theta^1\otimes Z_{\oneb} + (\nabla_{\oneb}\nabla^{1}f-iA_{\oneb}{}^{1}f)\theta^{\oneb}\otimes Z_{1}.
\end{equation}
An infinitesimal contact diffeomorphism $V=V_f$ is a CR symmetry if and only if $\DJ f=0$, we refer to this as the \emph{CR infinitesimal automorphism equation}.

\subsection{Abstract infinitesimal deformations of CR 3-manifolds}\label{subsec:Abstract-deformations}

Here we consider the space of infinitesimal deformations of a compact CR $3$-manifold $(M,H,J)$ up to equivalence, where two infinitesimal deformations of $(M,H,J)$ are equivalent if they are related by the linearized action of the diffeomorphism group of $M$. It is well known that it suffices to consider only deformations preserving the contact distribution $H$ on $M$, due to a famous result known as Gray's stability theorem:
\begin{lemma}[\cite{Gray1959}]\label{thm:Gray}
Let $(M,H_t)$, $t\in [0,1]$, be a family of contact structures on a compact manifold $M$, smooth in the sense that there is a smooth family of $1$-forms $\theta_t$ with $\ker\theta_t = H_t$. Then there exists a smooth path of diffeomorphisms $\varphi_t$ of $M$ such that $\varphi_0=\mathrm{id}$ and $\varphi_t:(M,H) \to (M,H_t)$ is a contact diffeomorphism for all $t\in[0,1]$.
\end{lemma}
We therefore restrict our consideration to the space of infinitesimal deformations of the CR manifold $(M,H,J)$ arising from a smooth $1$-parameter family of CR structures $J_t$ on $(M,H)$ with $J_0=J$. Let $(M,H,J_t)$, $t\in[0,\epsilon)$, be a such a smooth family of CR structures on $M$. Denoting $\left.\frac{d}{dt}\right|_{t=0}J_t$ by $\dot{J}$, differentiating the equation $J_t^2=-\mathrm{id}_{H}$ at $t=0$ we obtain
\begin{equation*}
\dot{J}J+J\dot{J}=0.
\end{equation*}
We let $\mathcal{D}ef(M)\subset \mathrm{End}(H)$ denote the bundle of consisting of endomorphisms of $H$ which anticommute with $J$; note that $\mathcal{D}ef(M)$ depends on $(M,H,J)$. The space of infinitesimal deformations of $(M,H,J)$, with $H$ held fixed, is then the space of smooth sections of $\mathcal{D}ef(M)$. That is, if $E$ is a section of $\mathcal{D}ef(M)$, then there is a path $J_t$ of CR structures on $(M,H)$ with $J_0=J$ such that $J_t = J + tE + O(t^2)$. In fact, if we write $E$ locally as $E_{1}{}^{\oneb}\theta^1\otimes Z_{\oneb} + E_{\oneb}{}^1\theta^{\oneb}\otimes Z_1$ then such a path is given by
\begin{equation*}
J_t = (1 + t^2 |E|^2)^{1/2} J + tE
\end{equation*}
where $|E|^2=E_1{}^{\oneb}E_{\oneb}{}^1$ \cite{ChengLee1990}. We refer to a section $E$ of $\mathcal{D}ef(M)$ as an \emph{infinitesimal deformation tensor} for $(M,H,J)$.

Given a smooth family $(M,H,J_t)$, $t\in[0,\epsilon)$, of CR structures we write $\mathbb{C}\otimes H = {}^t T^{1,0}\oplus {}^t T^{0,1}$ where $J_t$ acts on $^t T^{1,0}$ by $i$ and on $^t T^{0,1}$ by $-i$. If $Z_1$ is a local frame for $T^{1,0}= {}^0 T^{1,0}$, then (for sufficiently small $t$) there is a local frame $Z^t_1$ for $^t T^{1,0}$ given by
\begin{equation*}
Z^t_1 = Z_1 + \varphi_1{}^{\oneb}(t)Z_{\oneb}.
\end{equation*}
If we fix a contact form $\theta$ for $H$, and take the coframe $\{\theta, \theta^1_t, \theta^{\oneb}_t\}$ dual to $\{T,Z_1^t, Z_{\oneb}^t\}$ then
\begin{equation*}
\theta^1_t = \frac{1}{1-|\varphi(t)|^2}\left(\theta^1 - \varphi_{\oneb}{}^{1}(t)\theta^{\oneb}\right),
\end{equation*}
where $\varphi_{\oneb}{}^1(t)=\overline{\varphi_1{}^{\oneb}(t)}$ and $|\varphi(t)|^2=\varphi_1{}^{\oneb}(t)\varphi_{\oneb}{}^1(t)$. Writing $J_t = i\theta^1_t\otimes Z_1^t -i \theta^{\oneb}_t\otimes Z_{\oneb}^t$ and
\begin{equation*}
\varphi_1{}^{\oneb}(t) = t\varphi_1{}^{\oneb} + O(t^2)
\end{equation*}
one easily sees that the corresponding infinitesimal deformation tensor $E=\left.\frac{d}{dt}\right|_{t=0}J_t$ is given locally by
\begin{equation}\label{eqn:deformation-tensor-local}
E= 2i\varphi_{1}{}^{\oneb}\theta^1\otimes Z_{\oneb} - 2i\varphi_{\oneb}{}^1\theta^{\oneb}\otimes Z_1.
\end{equation}

The linearization at $(M,H,J)$ of the action of contact diffeomorphisms (by pullback) on the space of CR structures on $(M,H)$ is given by the map which sends an infinitesimal contact diffeomorphism $V$ and an infinitesimal deformation tensor $E$ to the infinitesimal deformation tensor $E+\mathcal{L}_V J$. We say that  pair of infinitesimal deformation tensors $E, E'$ are \emph{equivalent} if $E'-E$ lies in the image of $\DJ$ (recall that $\DJ f = -\frac{1}{2}\mathcal{L}_{V_f} J$, for $f\in \Gamma(TM/H)$). If $E$ lies in the image of $\DJ$ we call $E$ a \emph{trivial} infinitesimal deformation tensor.

For later use we observe that, in a weak sense, an equivalence between infinitesimal deformations can be integrated. Let $J_t, J'_t$, $t\in[0,\epsilon)$, be a pair of smooth paths of CR structures on $(M,H)$ with $J_0=J'_0=J$. We say that $J'_t$ is a \emph{contact reparametrization} of $J_t$ if there exists a smooth path $\varphi_t$, $t\in[0,\epsilon)$, of contact diffeomorphisms of $(M,H)$ such that $\varphi_0=\mathrm{id}$ and $\varphi_t^* J_t = J'_t$, $t\in[0,1)$.
\begin{lemma}\label{lem:contact-reparametrization-Jdot}
Let $J_t, J'_t$, $t\in[0,\epsilon)$, be a pair of smooth paths of CR structures on a compact contact manifold $(M,H)$ with $J_0=J'_0=J$. Suppose that the initial infinitesimal deformations $\dot{J}$ and $\dot{J}'$ are equivalent. Then there is a contact reparametrization $J_t''$ of $J_t$ such that $\dot{J}''=\dot{J}'$.
\end{lemma}
\begin{proof}
Since $\dot{J}$ and $\dot{J}'$ are equivalent, there exists an infinitesimal contact diffeomorphism $V$ for which  $\dot{J}'-\dot{J}=\mathcal{L}_V J$. Let $\varphi_t$ denote the flow of $V$, and let $J_t''=\varphi_t^*J_t$. Then
\begin{equation*}
\left.\frac{d}{dt}\right|_{t=0} J_t''=\left.\frac{d}{dt}\right|_{t=0} \varphi_t^*J_t = \varphi_0^*\left.\frac{d}{dt}\right|_{t=0} J_t + \left.\frac{d}{dt}\right|_{t=0} \varphi_t^*J_0= \dot{J} + \mathcal{L}_V J = \dot{J}'
\end{equation*}
as required.
\end{proof}

\subsection{Infinitesimal deformations of strictly pseudoconvex hypersurfaces}\label{subsec:Infinitesimal-hyerpsurface-deformations}

Let $M$ be a strictly pseudoconvex hypersurface in a complex surface $\Sigma$. Then $M$ carries an \emph{induced CR structure} $(M,H,J)$, where $H_p$ is the maximal complex subspace in $T_p M\subset T_p \Sigma$ for each $p\in M$ and $J$ is induced from the standard complex structure on $\Sigma$.
Since the considerations of this section will be local (and biholomorphically invariant) we will simply consider the case of a strictly pseudoconvex hypersurface in $\mathbb{C}^2$.
It is also no loss of generality to assume that our deformations are parametrized. Let $M$ be a strictly pseudoconvex hypersurface in $\mathbb{C}^2$, with induced CR structure $(M,H,J)$. We say that a smooth family of embeddings $\psi_t:M\to \mathbb{C}^2$, $t\in[0,\epsilon)$, is a \emph{parametrized deformation} of $M$ if $\psi_0 = \mathrm{id}_M$ and $\psi_t(M)$ is strictly pseudoconvex for all $t$. By pulling back the induced CR stuctures on $\psi_t(M)$ by $\psi_t$ for each $t$, one obtains a smooth family of CR structures $(M,H_t,J_t)$ on $M$ with $(M,H_0,J_0)=(M,H,J)$. We say that a parametrized deformation $\psi_t$ of $M$ is \emph{contact parametrized} if the induced family of CR structures $(M,H_t,J_t)$ on $M$ satisfies $H_t=H$ for all $t$, equivalently if $\psi_t:M\to \psi_t(M)$ is a contact diffeomorphism for all $t$, where the contact structure on $\psi_t(M)\subset \mathbb{C}^2$ comes from the induced CR structure. By Gray's stability theorem (\cref{thm:Gray}) any parametrized deformation may be reparametrized by a $1$-parameter family of diffeomorphisms of $M$ so that it becomes a contact parametrized deformation.

Given a strictly pseudoconvex hypersurface $M\subset\mathbb{C}^2$ it is usual to identify the real tangent space $T\mathbb{C}^2|_{M}$ with the space $T_{(1,0)}:=\mathbb{C}TM / T^{0,1}$ defined intrinsically in terms of the CR structure of $M$. Locally this identification is given by the map
\begin{equation*}
T_{(1,0)} \ni V^0\, T + V^1 Z_1 \; \mathrm{mod}\; T^{0,1}
\mapsto (\mathrm{Re}\, V^0)T + (\mathrm{Im}\, V^0) JT + 2\mathrm{Re}(V^1 Z_1) \in T\mathbb{C}^2|_{M},
\end{equation*}
where here $J$ denotes the standard complex structure on $\mathbb{C}^2$.
If $\psi_t$ is a contact parametrized deformation of $M$ then $\left.\frac{d}{dt}\right|_{t=0}\psi_t$ defines a section of $T\mathbb{C}^2|_M$ called the \emph{variational vector field}. We usually think of the variational vector field as a section of $T_{(1,0)}$ and denote it by $\dot{\psi}$. A contact Hamiltonian vector field $V$ on $M$ (taken mod $T^{0,1}$) is a trivial example of a variational vector field, since the flow of $V$ may be thought of as a (trivial) contact parametrized deformation of $M$. The variational vector field of a general contact parametrized deformation is in some sense a complex analog of a contact Hamiltonian vector field, as shown by the following lemma.
\begin{lemma}[\cite{BlandEpstein1996}]\label{lem:variational-vector}
Let $\psi_t:M\to \mathbb{C}^2$, $t\in[0,\epsilon)$, be a contact parametrized deformation of the strictly pseudoconvex hypersurface $M\subset \mathbb{C}^2$. The variational vector field $\dot{\psi}\in \Gamma(T_{(1,0)})$ is given locally with respect to an admissible coframe $\{\theta,\theta^1,\theta^{\oneb}\}$ by
\begin{equation*}
\dot{\psi} = fT + if^1Z_1 \; \mathrm{mod}\; T^{0,1}
\end{equation*}
where $f$ is the complex function $\theta(\dot{\psi})$ and $f^1=\nabla^1 f$. Moreover, if $(M,H, J_t)$ is the smooth family of CR structures on $M$ arising from $\psi_t$ then the initial infinitesimal deformation tensor $E=\dot{J}$ is given locally by \cref{eqn:deformation-tensor-local} with
\begin{equation*}
\varphi_{\oneb}{}^1 = -i(\nabla_{\oneb}\nabla^{1} f - i A_{\oneb}{}^{1}f), \quad \varphi_{1}{}^{\oneb} = \overline{\varphi_{\oneb}{}^1 }= i(\nabla_1\nabla^{\oneb} \bar{f} + i A_{1}{}^{\oneb}\bar{f}).
\end{equation*}
\end{lemma}
\begin{proof}
We write $\psi_t:M\to \mathbb{C}^2$ in components as $\psi_t=(\psi^1_t,\psi^2_t)$. Since, by definition, $\psi_t:(M,H_t,J_t)\to \mathbb{C}^2$ is a CR embedding for each $t\in [0,\epsilon)$, it follows that the component functions $\psi^1_t$ and $\psi^2_t$ are CR functions for $(M,H_t,J_t)$. Let $Z_1$ be a local frame for $T^{1,0}= {}^0 T^{1,0}$. Then (for sufficiently small $t$) there is a local frame for $^t T^{0,1}$ given by $Z_{\oneb}^t = Z_{\oneb} + \varphi_{\oneb}{}^1(t)Z_1$, with $\varphi_{\oneb}{}^1(t) = t\varphi_{\oneb}{}^1 + O(t^2)$. The fact that $\psi^1_t$ and $\psi^2_t$ are CR functions for $(M,H_t,J_t)$ is expressed by the equations $Z_{\oneb}^t\psi^k_t=0$, $k=1,2$. Differentiating these equations at $t=0$ we obtain
\begin{equation}\label{eqn:infinitesimal-deformation-equation}
Z_{\oneb}\dot{\psi}^k + \varphi_{\oneb}{}^1 Z_1 \psi^k=0, \quad k=1,2
\end{equation}
where $\dot{\psi}^k=\left.\frac{d}{dt}\right|_{t=0}\psi^k_t$, and $\psi^k=\psi^k_0$ (the $k^{th}$ component of the initial embedding). Writing $\dot{\psi}= f T + V^1 Z_1 \; \mathrm{mod}\; T^{0,1} \in \Gamma(T_{(1,0)})$, as a section of $T\mathbb{C}^2|_M$ we have
\begin{equation*}
\dot{\psi} = (\mathrm{Re}\, f)T + (\mathrm{Im}\, f) JT + V^1 Z_1 + V^{\oneb}Z_{\oneb}.
\end{equation*}
Letting $(z^1,z^2)$ denote the coordinates for $\mathbb{C}^2$ and evaluating $\dee z^k$ on the above display (noting that $\dee z^k(JT)=i\dee z^k(T)$ and $\dee z^k(Z_{\oneb})=0$) we obtain $\dot{\psi}^k = f \dee z^k(T) + V^1 \dee z^k(Z_1)$, $k=1,2$. If we think of the coordinates $z^k$ as maps from $\mathbb{C}^2$ to $\mathbb{C}$, then restricting to $M$ we have $\psi^k=z^k: M\to \mathbb{C}$. For a (real or complex) vector field $V$ tangent to $M$ we therefore have $\dee z^k(V)= V\psi^k$, and thus
\begin{equation*}
\dot{\psi}^k = f T\psi^k + V^1 Z_1\psi^k, \quad k=1,2.
\end{equation*}
Using that $Z_{\oneb}\psi^k=0$, $k=1,2$, we therefore have
\begin{align}\label{eqn:Z1bar-on-psik}
Z_{\oneb} \dot{\psi}^k &= (Z_{\oneb}f)T\psi^k + f Z_{\oneb}T\psi^k + (Z_{\oneb}V^1)Z_1\psi^k + V^1 Z_{\oneb}Z_1\psi^k\\
\nonumber &= (Z_{\oneb}f)T\psi^k + f [Z_{\oneb},T]\psi^k + (Z_{\oneb}V^1)Z_1\psi^k + V^1 [Z_{\oneb},Z_1]\psi^k.
\end{align}
From the structure equations \cref{eqn:h11-definition} and \cref{eqn:pseudohermitian-connection1} it is straightforward to compute that
\begin{equation*}
[Z_{\oneb},Z_1]= ih_{1\oneb}T + \omega_1{}^1(Z_{\oneb})Z_1 - \omega_{\oneb}{}^{\oneb}(Z_1)Z_{\oneb},
\quad \mathrm{and} \quad
[Z_{\oneb},T] = A^1{}_{\oneb}Z_1 - \omega_{\oneb}{}^{\oneb}(T)Z_{\oneb}.
\end{equation*}
Substituting \cref{eqn:Z1bar-on-psik} into \cref{eqn:infinitesimal-deformation-equation} we therefore obtain
\begin{equation*}
(Z_{\oneb}f + ih_{1\oneb}V^1)T\psi^k + (Z_{\oneb}V^1 + f A^1{}_{\oneb} + V^1\omega_1{}^1(Z_{\oneb}) + \varphi_{\oneb}{}^1)Z_1\psi^k =0,\quad k=1,2.
\end{equation*}
Recall that $\psi=(\psi^1,\psi^2)$ is a CR embedding, so that
\begin{equation*}
\det \begin{pmatrix} Z_1\psi^1 &Z_1 \psi^2\\ T\psi^1 & T\psi^2
\end{pmatrix}\neq 0
\end{equation*}
on $M$.
It follows that $Z_{\oneb}f + ih_{1\oneb}V^1=0$, so that $V^1=ih^{1\oneb}Z_{\oneb}f=i\nabla^1f$, which proves the first statement of the lemma. Similarly, it follows that
$Z_{\oneb}V^1 + f A^1{}_{\oneb} + V^1\omega_1{}^1(Z_{\oneb}) + \varphi_{\oneb}{}^1=0$ and hence
\begin{equation*}
\varphi_{\oneb}{}^1 = -\nabla_{\oneb}V^1 -fA^1{}_{\oneb} = -i\nabla_{\oneb}\nabla^1 f -fA^1{}_{\oneb}.
\end{equation*}
Conjugating this gives the formula for $\varphi_{1}{}^{\oneb}$, as claimed.
\end{proof}
The following lemma shows that the real part of $\theta(\dot{\psi})$ depends only on the contact parametrization of the smooth family $\psi_t(M)$ of strictly pseudoconvex hypersurfaces in $\mathbb{C}^2$, and can be taken to be zero.
\begin{lemma}\label{lem:making-psi-dot-imaginary}
Let $\psi_t:M\to \mathbb{C}^2$, $t\in[0,\epsilon)$, be a contact parametrized deformation of the compact strictly pseudoconvex hypersurface $M\subset \mathbb{C}^2$. Then there is a smooth family $\varphi_t$, $t\in[0,\epsilon)$, of contact diffeomorphisms of $M$ which reparametrizes $\psi_t$ to $\psi'_t = \psi_t\circ \varphi_t$ such that $\theta(\dot{\psi'})$ is imaginary, where $\theta$ is any pseudohermitian structure $\theta$ on $M$.
\end{lemma}
\begin{proof}
Fix any pseudohermitian structure $\theta$ for the initial CR structure $(M,H,J)$, and let $f=\theta(\dot{\psi})$.
Then $\left.\frac{d}{dt}\right|_{t=0}\psi_t =  (\mathrm{Re}\, f)T + (\mathrm{Im}\, f) JT \; \mathrm{mod}\; \mathbb{C}\otimes H$,
where here $J$ denotes the standard complex structure on $\mathbb{C}^2$. Let $V$ be the contact Hamiltonian vector field with potential $-\mathrm{Re}\,f$, and let $\varphi_t$ be the flow of $V$. Then $\left.\frac{d}{dt}\right|_{t=0}\varphi_t = -(\mathrm{Re}\, f)T \; \mathrm{mod}\; H$. Since $\psi_0=\varphi_0=\mathrm{id}_M$ we have $\left.\frac{d}{dt}\right|_{t=0}\psi'_t = \left.\frac{d}{dt}\right|_{t=0}\psi_t+\left.\frac{d}{dt}\right|_{t=0}\varphi_t$ and the result follows.
\end{proof}

\section{Variation of the Burns-Epstein invariant}\label{sec:Burns-Epstein}

In \cite{BurnsEpstein1988} Burns and Epstein defined a global invariant of compact CR $3$-manifolds whose holomorphic tangent bundle is trivial, by analogy with the Chern-Simons invariant for a (conformal) Riemannian $3$-manifold. Let $\mathcal{G}\to M$ denote the CR Cartan structure bundle of $(M,H,J)$.  ($\mathcal{G}$ is denoted $Y$ in \cite{ChernMoser1974}.) Let $\omega \in \Omega^1(\mathcal{G},\mathfrak{su}(2,1))$ denote the CR Cartan connection, and let $K=\dee \omega + \omega\wedge \omega$ denote its curvature. The Chern form $c_2(K)=\frac{1}{8\pi^2}\mathrm{tr}(K\wedge K)$ is a closed, basic $4$-form on $\mathcal{G}$. If
\begin{equation*}
Tc_2(\omega)=\frac{1}{8\pi^2}(\omega\wedge K + \frac{1}{3}\omega\wedge\omega\wedge\omega)
\end{equation*}
then from the definition of $K$ one has $\dee\, Tc_2(\omega) = c_2(K)$. But $c_2(K)$ vanishes since $\dim M=3$. So $Tc_2(\omega)$ is a closed, CR invariant $3$-form on $\mathcal{G}$. The bundle $\mathcal{G}\to M$ is a trivial extension of the frame bundle $\mathcal{G}_0\to M$ of the holomorphic tangent bundle $T^{1,0}$, $\mathcal{G}\cong \mathcal{G}_0\times \mathbb{H}_1$, where $\mathbb{H}_1$ is the real $3$-dimensional Heisenberg group. It follows that $\mathcal{G}\to M$ admits global sections if and only if the holomorphic tangent bundle $T^{1,0}$ is trivial. In \cite{BurnsEpstein1988} Burns and Epstein showed that if one pulls $Tc_2(\omega)$ back to $M$ via two different sections of $\mathcal{G}$ corresponding to global admissible coframes with $h_{1\oneb}=1$ as in \cite{Webster1978}, then the resulting $3$-forms on $M$ differ by an exact form. Given any such section $\varsigma:M\to\mathcal{G}$, one may define
\begin{equation*}
\mu = \mu(M) := \int_M \varsigma^* Tc_2(\omega),
\end{equation*}
which is then a global CR invariant, known as the \emph{Burns-Epstein invariant} of $(M,H,J)$. As remarked in \cite{BurnsEpstein1988}, if one only assumes that $c_1(T^{1,0})$ is zero in $H^2(M,\mathbb{R})$ (i.e. one allows $c_1(T^{1,0})$ to be a torsion class) then there is some $k\in\mathbb{N}$ such that $(T^{1,0})^k$ is trivial, so one may take a $k$-fold multi-section of $\mathcal{G}$ and integrate $\frac{1}{k}Tc_2(\omega)$ over the image to define $\mu$. For further extensions of this invariant, including to higher dimensions, see \cite{BiquardHerzlich2005, BurnsEpstein1990a, ChengLee1990, Marugame2016}.

In $3$-dimensions, the total $Q'$-curvature of a pseudo-Einstein CR manifold is a scalar multiple of the Burns-Epstein invariant. In this case, the pseudohermitian structure given by $\theta$ is said to be \emph{pseudo-Einstein} if $\nabla_1 R - i \nabla^1A_{11}=0$; if $\theta$ is a pseudo-Einstein contact form then
$Q'=\Delta_b R+\frac{1}{2}R^2-2|A|^2$,
where $\Delta_b$ is the sub-Laplacian, $|A|^2=A_{11}A^{11}$, and the Burns-Epstein invariant is given by \cite{CaseYang2013, Hirachi2014}
\begin{equation*}
\mu=-\frac{1}{8\pi^2}\int_M Q'\,\theta\wedge\dee\theta.
 \end{equation*} The total $Q'$-curvature, $\overline{Q'}=\int_M Q'\,\theta\wedge\dee\theta$, is known to give a different generalization of the Burns-Epstein invariant to higher dimensions.

As with the contact distribution, there is no loss of generality in holding the Cartan structure bundle $\mathcal{G}\to M$ fixed when considering deformations of CR structures on $M$. A smooth family $(M,H,J_t)$, $t\in [0,\epsilon)$, of CR structures on $M$ gives rise to a corresponding family $\omega_t$ of Cartan connections on $\mathcal{G}$. Letting $\mu_t$ denote the Burns-Epstein invariant corresponding to $J_t$ and $\dot{\omega}=\left.\frac{d}{dt}\right|_{t=0} \omega_t$, from the definition of $Tc_2$ and of $K$ one obtains \cite[Proposition 3.3]{BurnsEpstein1988}
\begin{equation*}
\left.\frac{d}{dt}\right|_{t=0}\mu_t = -\frac{1}{4\pi^2}\int_M \mathrm{tr}\,(\dot{\omega}\wedge K),
\end{equation*}
where $K$ is the curvature of $\omega=\omega_0$ (and we are implicitly pulling back the integrand by a section of $\mathcal{G}\to M$). Using the framework of \cite{Webster1978}, Burns and Epstein then compute (see also \cite[Proposition 2.6]{ChengLee1990}) that if $E=\left.\frac{d}{dt}\right|_{t=0} J_t$ is given locally in terms of $\varphi_1{}^{\oneb}$ by \cref{eqn:deformation-tensor-local} then
\begin{equation}\label{eqn:Burns-Epstein-first-variation}
\left.\frac{d}{dt}\right|_{t=0}\mu_t = \frac{1}{4\pi^2}\int_M (\varphi_1{}^{\oneb}Q_{\oneb}{}^1 + \varphi_{\oneb}{}^1Q_1{}^{\oneb})\,\theta\wedge \dee\theta.
\end{equation}
In particular we see that the only critical points of $\mu$ are the locally spherical CR structures.

For a compact strictly pseudoconvex hypersurface in $\mathbb{C}^2$, the real first Chern class of the holomorphic tangent bundle is the zero class, so the Burns-Epstein invariant is always defined. As in \cite{Hirachi2014} (cf.\ \cite[Theorem 1.2]{HirachiMarugameMatsumoto2017}) we observe:
\begin{lemma}\label{lem:Burns-Epstein-infinitesimal-wiggle}
Let $\psi_t:M\to \mathbb{C}^2$, $t\in[0,\epsilon)$, be a contact parametrized deformation of the compact strictly pseudoconvex hypersurface $M\subset \mathbb{C}^2$, and let $\mu_t$ be the Burns-Epstein invariant of $\psi_t(M)\subset \mathbb{C}^2$. Then
\begin{equation*}
\left.\frac{d}{dt}\right|_{t=0}\mu_t = \frac{3}{2\pi^2}\,\mathrm{Im}\int_M f\mathcal{O}
\end{equation*}
where $f=\btheta(\dot{\psi})$, and the integrand is regarded as a density.
\end{lemma}
\begin{proof}
Let $(M,H,J_t)$ denote the CR structure on $M$ obtained by pulling back the CR structure on the strictly pseudoconvex hypersurface $\psi_t(M)\subset \mathbb{C}^2$ via $\psi_t$. Let $(\theta,\theta^1,\theta^{\oneb})$ be an admissible coframe, which we can assume, without loss of generality, to be global. Let $E=\left.\frac{d}{dt}\right|_{t=0}J_t$ be given in terms of $\varphi_1{}^{\oneb}$ by \cref{eqn:deformation-tensor-local}. Then by \cref{lem:variational-vector} we have $\varphi_{\oneb}{}^{1} = -i(\nabla_{\oneb}\nabla^{1}  f - i A_{\oneb}{}^{1} f)$, where $f$ is thought of as a function using the trivialization of $\cE_{\mathbb{R}}(1,1)$ induced by $\theta$.
Then, using \cref{eqn:Burns-Epstein-first-variation} and integrating by parts (using (2.18) of \cite{Lee1988}, a well known consequence of Stokes' theorem)
\begin{align*}
\left.\frac{d}{dt}\right|_{t=0}\mu_t &= \frac{1}{2\pi^2}\, \mathrm{Re}\int_M \varphi^{11}Q_{11}\,\theta\wedge \dee\theta\\
&= \frac{1}{2\pi^2}\, \mathrm{Re}\int_M -i(\nabla^1\nabla^1 f - i A^{11}f)Q_{11}\,\theta\wedge \dee\theta\\
&= \frac{1}{2\pi^2}\, \mathrm{Re}\int_M -if(\nabla^1\nabla^1 Q_{11}- i A^{11}Q_{11})\,\theta\wedge \dee\theta\\
&= \frac{3}{2\pi^2}\, \mathrm{Im}\int_M f\mathcal{O}\,\theta\wedge \dee\theta,
\end{align*}
as required.
\end{proof}
\begin{remark}
Note that, since $\mathcal{O}$ is real, the formula for $\left.\frac{d}{dt}\right|_{t=0}\mu_t$ in \cref{lem:Burns-Epstein-infinitesimal-wiggle} depends only on the imaginary part of $f$. This makes sense, as $\mathrm{Re}\,f$ depends on the particular contact parametrization of the deformation (cf.\ \cref{lem:making-psi-dot-imaginary}), whereas $\mu_t$ depends only on the family of strictly pseudoconvex hypersurfaces $\psi_t(M)\subset \mathbb{C}^2$.
\end{remark}

In order to prove \cref{thm:DeformationRigidity} we will use:
\begin{theorem}[\cite{ChengLee1995}]\label{thm:ChengLee1995}
The CR 3-sphere is a strict local minimizer for the Burns-Epstein invariant.
\end{theorem}
\cref{thm:ChengLee1995} represents the culmination of the work of Ch\^eng and Lee in \cite{ChengLee1990, ChengLee1995}. In \cite{ChengLee1990} it was established that the second variation of the Burns-Epstein invariant at the standard CR 3-sphere is positive definite for infinitesimal deformations orthogonal to the orbit of the contact diffeomorphism group (as here the deformations are taken to fix the underlying contact structure). In \cite{ChengLee1995} a local slice theorem for the space of CR structures on $(M,H)$ under the action of the contact diffeomorphism group was established and used to prove that (in the slice) the second variation of the Burns-Epstein invariant at the sphere gives a suitably good approximation to the Burns-Epstein invariant to imply that the sphere is a strict local minimizer.

\begin{proof}[Proof of \cref{thm:DeformationRigidity}]
Let $\Omega_t$, $t\in [0,1]$, be a smooth family of smooth bounded strictly pseudoconvex domains in $\mathbb{C}^2$, with $\Omega_0$ being the unit ball. Assume $\partial\Omega_t$ is obstruction flat for all $t$. By \cref{lem:Burns-Epstein-infinitesimal-wiggle}, the Burns-Epstein invariant $\mu(\partial\Omega_t)$  then remains constant for all $t$ and, hence, $\mu(\partial\Omega_t)=\mu(\partial\Omega_0)$. But the Burns-Epstein invariant is a strict local minimizer for the CR $3$-sphere $\partial\Omega_0$ by the Ch\^eng and Lee result \cref{thm:ChengLee1995}. Hence $\partial\Omega_t$ must be globally CR equivalent to the unit sphere $\partial\Omega_0$, for all $t$. Since $\Omega_0$ is simply connected with connected boundary, by continuity the same must be true for each $\Omega_t$. It then follows by the Riemann mapping theorem of \cite{ChernJi1996} that each $\Omega_t$ is biholomorphic to the unit ball $\Omega_0$.
\end{proof}

\begin{remark}
As remarked after the statement of \cref{thm:DeformationRigidity}, the conclusion of the theorem can be easily improved to state that there exists a smooth family $\Phi_t:\Omega_t\to\Omega_0$ of biholomorphisms, $t\in [0,1]$. To see this, note that one may contact parametrize the boundary deformation and pull the boundary CR structures back to $S^3$, giving a smooth family of CR structures with fixed underlying contact distribution. Since each of these CR structures must be spherical, the slice theorem of Ch\^eng and Lee \cite{ChengLee1995} says that one may find a smooth family of contact diffeomorphism of $S^3$ parametrizing them. One can then use these to reparametrize the contact parametrization of the deformation so that it becomes a parametrization by CR diffeomorphisms. This parametrization then extends to the domains as a parametrization by biholomorphisms.
\end{remark}

\section{A deformation complex on the CR 3-sphere}\label{sec:DeformationComplex}

Here we describe in detail the deformation complex \cref{eqn:DeformationComplex} on the CR $3$-sphere, and use it to prove \cref{thm:FormalDeformations}.

\subsection{The linearized curvature operator on the sphere}\label{subsec:Linearized-Q11-at-sphere}

Let $(M,H,J)$ be a compact CR $3$-manifold. We have already introduced the operator $-2\DJ:\cE_{\mathbb{R}}(1,1)\to \mathcal{D}ef(M)$ which gives the infinitesimal deformation tensor $E=-2\DJ f=\mathcal{L}_{V_f} J$ arising from pulling back the CR structure by the flow of the contact Hamiltonian vector field $V_f$ with potential $f$. Being a weight $(-2,-2)$ CR invariant, by \cref{eqn:Q-endomorphism} the Cartan curvature $Q$ of $(M,H,J)$ is a section of the bundle $\mathcal{C}urv(M)= \mathcal{D}ef(M)\otimes \cE_{\mathbb{R}}(-2,-2)$. If $(M,H,J_t)$, $t\in[0,\epsilon)$, is a smooth family of CR structures on $M$ with $J_0=J$, then $\dot{Q}=\left.\frac{d}{dt}\right|_{t=0}Q_t$ will not in general be a section of $\mathcal{C}urv(M)$, but only of the larger bundle $\mathrm{End}(H)\otimes \cE_{\mathbb{R}}(-2,-2)$. If $(M,H,J)$ is locally spherical, however, then $\dot{Q}\in \Gamma(\mathcal{C}urv(M))$ since the vanishing of $Q$ implies locally
\begin{equation*}
\dot{Q}=i\dot{Q}_1{}^{\oneb}\theta^1\otimes Z_{\oneb} - i\dot{Q}_{\oneb}{}^1\theta^{\oneb}\otimes Z_1.
\end{equation*}
For the CR $3$-sphere, we denote the CR invariant linearized curvature operator $\mathcal{D}ef(M)\to\mathcal{C}urv(M)$ by $\RJ$. We will explicitly compute this operator in terms of the standard pseudohermitian structure on $S^3$.

Let $S^3$ denote the unit sphere in $\mathbb{C}^2$. Let $(z,w)$ be the standard coordinates on $\mathbb{C}^2$ and let $u=1-|z|^2-|w|^2$. The standard pseudohermitian structure on $S^3$ is given by taking $\theta = i\partial u |_{TS^3}$. A global framing of the holomorphic tangent bundle of $S^3$ is given by $Z_1 = \bar{w}\frac{\partial}{\partial z} - \bar{z}\frac{\partial}{\partial w}$, giving $\theta^1=w\dee z - z\dee w$. All forms written in ambient coordinates are implicitly pulled back to $S^3$. The contact form $\theta$ may be written as $i(z\dee \bar{z} + w\dee \bar{w})$ and we have $\dee \theta = i(\dee z \wedge \dee\bar{z} + \dee w \wedge \dee\bar{w})$. One then sees that $\dee\theta = i\theta^1\wedge\theta^{\oneb}$, e.g., by computing the two $2$-forms with respect to the frame $\{T,Z_1,Z_{\oneb}\}$. Solving the structure equation \cref{eqn:pseudohermitian-connection1} gives $\omega_1{}^1=-2i\theta$, and $A^1{}_{\oneb}=0$. Therefore $\dee \omega_1{}^1 = -2i\dee \theta = 2\theta^1\wedge \theta^{\oneb}$, so $R=2$. We now consider a smooth $1$-parameter family of CR structures on $S^3$, given in terms of the vector field
\begin{equation*}
Z_1^t = \frac{1}{\sqrt{1-|\varphi(t)|^2}}(Z_1 + \varphi_1{}^{\oneb}(t)Z_{\oneb})
\end{equation*}
spanning $^t T^{1,0}$, with $\varphi_1{}^{\oneb}(0)=0$. Then
\begin{equation*}
\theta^1_t = \frac{1}{\sqrt{1-|\varphi(t)|^2}}(\theta^1 - \varphi_{\oneb}{}^1(t)\theta^{\oneb})
\end{equation*}
where $\varphi_{\oneb}{}^1(t)=\overline{\varphi_1{}^{\oneb}(t)}$, and because of our choice of normalization we have $\theta^1_t\wedge\theta^{\oneb}_t = \theta^1\wedge\theta^{\oneb}$, so that $h_{1\oneb}(t)=1$. We write $\varphi_{\oneb}{}^1(t) = t\varphi_{\oneb}{}^1 + O(t^2)$ so that $\dot{\theta}^1=\left.\frac{d}{dt}\right|_{t=0}\theta^1_t=-\varphi_{\oneb}{}^1\theta^{\oneb}$. Writing \cref{eqn:pseudohermitian-connection1} for each $t$ as
\begin{equation*}
\dee\theta^1_t = \theta^1_t \wedge \omega_{1}{}^1(t) + A^1{}_{\oneb}(t)\theta\wedge\theta^{\oneb}_t
\end{equation*}
and differentiating with respect to $t$ at $t=0$ we have
\begin{equation*}
\dee \dot{\theta^1} = \dot{\theta}^1 \wedge \omega_1{}^1 + \theta^1 \wedge \dot{\omega}_1{}^1 + \dot{A}^1{}_{\oneb}\theta\wedge\theta^{\oneb}.
\end{equation*}
Since $\dot{\theta^1}=-\varphi_{\oneb}{}^1\theta^{\oneb}$, we also have $\dee \dot{\theta^1} = -(\dee \varphi_{\oneb}{}^1)\wedge\theta^1-\varphi_{\oneb}{}^1\dee\theta^{\oneb}$, and equating this with the above display we obtain that
\begin{equation}
\dot{A}^1{}_{\oneb}=-\nabla_0\varphi_{\oneb}{}^1
\end{equation}
and $\dot{\omega}_1{}^1 = -(\nabla_1\varphi_{\oneb}{}^1)\theta^{\oneb} \;\mathrm{mod}\; \theta^1$. But $\dot{\omega}_1{}^1$ is imaginary since $h_{1\oneb}(t)=1$ for all $t$, so
\begin{equation}
\dot{\omega}_1{}^1 = (\nabla_{\oneb}\varphi_1{}^{\oneb})\theta^1-(\nabla_1\varphi_{\oneb}{}^1)\theta^{\oneb}.
\end{equation}
Since $\dee \dot{\omega}_1{}^1 = \dot{R}\theta^1\wedge\theta^{\oneb}\;\mathrm{mod}\; \theta^1\wedge\theta, \theta^{\oneb}\wedge\theta$, we therefore obtain
\begin{equation}
\dot{R} = \nabla^1\nabla^1\varphi_{11} - \nabla^{\oneb}\nabla^{\oneb}\varphi_{\oneb\oneb}.
\end{equation}
We may now easily compute $\dot{Q}=\left.\frac{d}{dt}\right|_{t=0} Q_t$. Since $R=2$ we have
\begin{align*}
\left.\frac{d}{dt}\right|_{t=0} \nabla^t_1 \nabla^t_1 R_t & = \left.\frac{d}{dt}\right|_{t=0} (Z^t_1-Z^t_1\,\hook\,\omega_1{}^1(t))Z_1^t R_t\\
& = Z_1Z_1 \dot{R} + Z_1(\varphi_1{}^{\oneb}Z_{\oneb}R) + (\varphi_1{}^{\oneb}Z_{\oneb} - Z_1\,\hook\,\dot{\omega}_1{}^1)Z_1 R\\
& = \nabla_1\nabla_1 \dot{R}.
\end{align*}
Similarly, since $A_{11}=0$ we have
\begin{align*}
\left.\frac{d}{dt}\right|_{t=0} \nabla^t_0 A_{11}(t) &= \nabla_0 \dot{A}_{11} = -\nabla_0\nabla_0 \varphi_{11}\,; \\
\left.\frac{d}{dt}\right|_{t=0} \nabla_1^t\nabla^1_t A_{11}(t) &= \nabla_1\nabla^1 \dot{A}_{11}=-\nabla_1\nabla^1\nabla_0 \varphi_{11}.
\end{align*}
From \cref{eqn:Cartan-umbilical-tensor} we therefore obtain (cf.\ \cite{ChengLee1990})
\begin{equation}
\dot{Q}_{11}
= -\frac{1}{6}(\varphi_{11,}{}^{11}{}_{11} - \varphi_{\oneb\oneb,}{}^{\oneb\oneb}{}_{11})
 - \varphi_{11,00} - \frac{2i}{3} \varphi_{11, 0}{}^1{}_1 + \frac{i}{2}R\,\varphi_{11,0}
\end{equation}
where indices placed after a comma denote covariant derivatives, so, e.g., $\varphi_{11,}{}^{11}{}_{11}$ denotes $\nabla_1\nabla_1\nabla^1\nabla^1\varphi_{11}$. While $R=2$ here, we have retained $R$ in the expression to emphasize that the pseudohermitian curvature shows up in the last term. The fact that there is only one term in the above expression involving $\varphi_{\oneb\oneb}$ (as opposed to its conjugate $\varphi_{11}$) will be exploited in the proof of \cref{thm:FormalDeformations}. The above computation shows that if $E=2i\varphi_1{}^{\oneb}\theta^1\otimes Z_{\oneb}-2i\varphi_{\oneb}{}^1\theta^{\oneb}\otimes Z_1$ is an infinitesimal deformation tensor, then applying the CR invariant linearized curvature operator $\RJ\colon \mathcal{D}ef(M)\to\mathcal{C}urv(M)$ we obtain
\begin{equation}\label{eqn:RJ-formula-part1}
\RJ E = iF_1{}^{\oneb}\theta^1\otimes Z_{\oneb}-iF_{\oneb}{}^1\theta^{\oneb}\otimes Z_1
\end{equation}
where
\begin{equation}\label{eqn:RJ-formula-part2}
F_{11}
= -\frac{1}{6}(\varphi_{11,}{}^{11}{}_{11} - \varphi_{\oneb\oneb,}{}^{\oneb\oneb}{}_{11}) - \varphi_{11,00} - \frac{2i}{3} \varphi_{11, 0}{}^1{}_1 + \frac{i}{2}R\,\varphi_{11,0}.
\end{equation}
Recalling that for every infinitesimal deformation tensor $E$ there is a family $J_t$ with $J_0=J$ and $\dot{J}=E$, we see that if $F=\RJ E$ for some $E$ then there is a family $J_t$ with $J_0=J$ and $\dot{Q}=F$.

\subsection{Deformations and the Bianchi identity}\label{subsec:Deformations-Bianchi}

Let $(M,H,J)$ be a compact CR $3$-manifold. The last operator we need to consider is simply the adjoint of the operator $\DJ:\cE_{\mathbb{R}}(1,1)\to \mathcal{D}ef(M)$. It is natural to define the adjoint with respect to the tautological weight $(2,2)$ volume form on $(M,H,J)$ coming from the identification of $\cE_{\mathbb{R}}(-2,-2)$ with $\Lambda^3$. Fixing a contact form $\theta$ for $H$ one may also take the adjoint of $\DJ$ with respect to the volume form $\theta\wedge\dee\theta$, and this gives the same result after we trivialize the relevant density bundles using $\theta$. The advantage of the CR invariant construction is that it turns out to give a CR invariant operator
\begin{equation*}
\DJad :\mathcal{C}urv(M)\to \cE_{\mathbb{R}}(-3,-3).
\end{equation*}
Moreover, the Bianchi identity \cref{eqn:Q11-Bianchi} turns out to be equivalent to $\DJad Q =0$ (cf.\ \cite{ChengLee1990}). This explains the use of the notation $\mathcal{B}ian(M)$ for $\cE_{\mathbb{R}}(-3,-3)$ in \cref{eqn:DeformationComplex}.

To formalize these observations, we define a local pairing between sections of $\mathcal{D}ef(M)$ and $\mathcal{C}urv(M)$ by
\begin{equation*}
\langle E, F\rangle = E^{11}F_{11} + E_{11}F^{11} = 2\mathrm{Re} ( E^{11}F_{11} )
\end{equation*}
where $E = E_1{}^{\oneb}\theta^1\otimes Z_{\oneb}+E_{\oneb}{}^1\theta^{\oneb}\otimes Z_1$ and $F= F_1{}^{\oneb}\theta^1\otimes Z_{\oneb}+F_{\oneb}{}^1\theta^{\oneb}\otimes Z_1$. Note that $\langle E, F\rangle$ has weight $(-2,-2)$ and may therefore be integrated. We define a natural CR invariant global pairing between sections of $\mathcal{D}ef(M)$ and $\mathcal{C}urv(M)$ by
\begin{equation*}
(E,F) = \int_M \langle E, F\rangle.
\end{equation*}
We define $\DJad:\mathcal{C}urv(M)\to \cE_{\mathbb{R}}(-3,-3)$ to be the adjoint of $\DJ$ with respect to this global pairing. Since in a local frame the component $E^{11}$ of $E=\DJ f$ is given by $\nabla^1\nabla^1f-iA^{11}f$, integrating by parts gives
\begin{equation}\label{eqn:DJad-formula}
\DJad F = 2\mathrm{Re}(\nabla^1\nabla^1F_{11}-iA^{11}F_{11})
\end{equation}
where $F$ is given locally by $F_1{}^{\oneb}\theta^1\otimes Z_{\oneb}+F_{\oneb}{}^1\theta^{\oneb}\otimes Z_1$. Recalling \cref{eqn:Q-endomorphism} we see that the Bianchi identity \cref{eqn:Q11-Bianchi} is equivalent to $\DJad Q =0$.

Now let $(S^3,H,J)$ be the standard CR structure on $S^3$ and let $J_t$, $t\in[0,\epsilon)$, be a smooth family of CR structures on $(S^3,H)$ with $J_0=J$. For each $t$ the Cartan umbilicity tensor $Q_t$ of $J_t$ satisfies the Bianchi identity $D^*_{J_t} Q_t =0$. Let $\dot{Q}=\left.\frac{d}{dt}\right|_{t=0}Q_t$. Differentiating the Bianchi identity $D^*_{J_t} Q_t =0$ with respect to $t$ at $t=0$ we obtain
\begin{equation}\label{eqn:linearized-Bianchi}
\DJad \dot{Q} =0,
\end{equation}
which can be seen as a Bianchi identity for $\dot{Q}$.

\subsection{The deformation complex}\label{subsec:deformation-complex}

The discussions in \cref{subsec:Abstract-deformations,subsec:Linearized-Q11-at-sphere,subsec:Deformations-Bianchi} allow us to conclude that on the CR $3$-sphere the following sequence of differential operators is a differential complex:
\begin{equation}\label{eqn:DeformationComplex2}
0\longrightarrow \cE_{\mathbb{R}}(1,1) \overset{\DJ}{\longrightarrow} \mathcal{D}ef(M) \overset{\RJ}{\longrightarrow} \mathcal{C}urv(M) \overset{\DJad}{\longrightarrow} \mathcal{B}ian(M) \longrightarrow 0.
\end{equation}
That $\RJ \DJ=0$ follows from the diffeomorphism invariance of the Cartan umbilicity tensor $Q$, and the fact that $Q$ vanishes for the standard CR $3$-sphere. That $\DJad\,\RJ=0$ follows from \cref{eqn:linearized-Bianchi}. A fact which is not obvious from the previous discussion is that the above complex is locally exact. This is because, for natural reasons \cite{Cap2008}, \cref{eqn:DeformationComplex2} turns out to be the Bernstein-Gelfand-Gelfand (BGG) complex on $S^3=\mathrm{SU}(2,1)/P$ corresponding to the adjoint representation of $\mathrm{SU}(2,1)$, which gives a fine resolution of the sheaf of constant sections of the homogeneous vector bundle $\cA= \mathrm{SU}(2,1)/P \times_P \mathfrak{su}(2,1) \to S^3$. The rank $8$ vector bundle $\cA\to S^3$ is flat and trivial, and the global constant sections are in one to one correspondence with the infinitesimal symmetries of the CR $3$-sphere. Moreover, the cohomology of \cref{eqn:DeformationComplex2} is equal to the cohomology of the de Rham complex on $S^3$ twisted by $\mathcal{A}$, which is $H^*(S^3,\mathbb{R})\otimes \mathfrak{su}(2,1)$. In particular,
\begin{equation}\label{eqn:exactness-in-deformation-complex}
\mathrm{ker}\, \RJ = \mathrm{im}\, \DJ \quad \text{ and } \quad \mathrm{ker}\, \DJad = \mathrm{im}\, \RJ .
\end{equation}

Let $\mathcal{E}\to S^3$ denote the trivial complex line bundle on $S^3$. Complexifying and type decomposing \cref{eqn:DeformationComplex2} one obtains the bigraded complex
\begin{equation}\label{eqn:bigraded-deformation-complex}
\xymatrix@C+0.5pc@R-1pc{
& (T^{1,0})^*\otimes T^{0,1} \ar[r]^(.5){\mathcal{R}^{\natural}} \ar[rdd]_<<<{\mathcal{R}^{-}} & (T^{1,0})^*\otimes T^{0,1} \ar[rd]^(.65){\mathcal{D}^*} & \\
\cE \ar[ru]^(.35){\mathcal{D}} \ar[rd]_(.35){\bar{\mathcal{D}}} & & & \cE\\
& (T^{0,1})^*\otimes T^{1,0} \ar[ruu]^<<<{\mathcal{R}^+} \ar[r]_(.5){\bar{\mathcal{R}}^{\natural}} &(T^{0,1})^*\otimes T^{1,0}\ar[ru]_(.65){\bar{\mathcal{D}}^*} &
}
\end{equation}
where we have suppressed the density weights in the notation. The initial bundle is really $\cE(1,1)$, the final bundle is $\cE(-3,-3)$, and in their second appearance the bundles $(T^{1,0})^*\otimes T^{0,1}$ and $(T^{0,1})^*\otimes T^{1,0}$ should be tensored with $\cE(-2,-2)$. By \cref{eqn:DJ-formula} the operators $\mathcal{D}$ and $\bar{\mathcal{D}}$ are given locally with respect to any admissible coframe $\{\theta,\theta^1,\theta^{\oneb}\}$ by
\begin{equation}\label{eqn:calD-formula}
\mathcal{D} f = (\nabla_1\nabla^{\oneb}f+iA_1{}^{\oneb}f) \theta^1\otimes Z_{\oneb},
\quad \text{and} \quad
\bar{\mathcal{D}}f = (\nabla_{\oneb}\nabla^{1}f-iA_{\oneb}{}^{1}f)\theta^{\oneb}\otimes Z_{1}.
\end{equation}
By definition we have $\DJ= \mathcal{D}+\bar{\mathcal{D}}$. By \cref{eqn:DJad-formula} the operators $\mathcal{D}^*$ and $\bar{\mathcal{D}}^*$ are given locally with respect to any admissible coframe by
\begin{equation}\label{eqn:calD-adjoint-formula}
\mathcal{D}^*(F_1{}^{\oneb}\theta^1\otimes Z_{\oneb}) = \nabla^1\nabla_{\oneb}F_1{}^{\oneb} -iA^1{}_{\oneb}F_1{}^{\oneb},
\quad \text{and} \quad
\bar{\mathcal{D}}^*(F_{\oneb}{}^{1}\theta^1\otimes Z_{\oneb}) = \nabla^{\oneb}\nabla_{1}F_{\oneb}{}^{1} +iA^{\oneb}{}_{1}F_{\oneb}{}^{1}.
\end{equation}
By definition $\DJad = \mathcal{D}^*+\bar{\mathcal{D}}^*$, where $\mathcal{D}^*$ is extended to act by zero on $(T^{0,1})^*\otimes T^{1,0}\otimes \cE(-2,-2)$, and similarly $\bar{\mathcal{D}}^*$ is extended in the obvious way. In \cref{subsec:Linearized-Q11-at-sphere} we derived a fairly simple formula for $\RJ$ with respect to standard pseudohermitian structure $\theta$ on the CR $3$-sphere, extended to a standard admissible coframe $\{\theta,\theta^1,\theta^{\oneb}\}$. The simplification comes from the fact that the pseudohermitian curvature $R$ is then constant, and the pseudohermitian torsion vanishes. By \cref{eqn:RJ-formula-part1} and \cref{eqn:RJ-formula-part2}, with respect to the standard admissible coframe on $S^3$ we have
\begin{align}
\mathcal{R}^{\natural}(2i\varphi_1{}^{\oneb}\theta^1\otimes Z_{\oneb}) &= i\left (-\frac{1}{6}\varphi_{1}{}^{\oneb}{}_{,}{}^{11}{}_{11} - \varphi_{1}{}^{\oneb}{}_{,00} - \frac{2i}{3} \varphi_{1}{}^{\oneb}{}_{,0}{}^1{}_1 + \frac{i}{2}R\,\varphi_{1}{}^{\oneb}{}_{,0}\right)\,\theta^1\otimes Z_{\oneb}\,;
\label{eqn:calR-natural}\\
\mathcal{R}^+(2i\varphi_{\oneb}{}^{1}\theta^{\oneb}\otimes Z_1) &= -\frac{i}{6}\varphi_{\oneb\oneb,}{}^{\oneb\oneb}{}_{1}{}^{\oneb}\,\theta^1\otimes Z_{\oneb}.
\label{eqn:calR-plus}
\end{align}
The remaining operators are conjugates of these, $\bar{\mathcal{R}}^{\natural} = \overline{\mathcal{R}^{\natural}}$ and $\mathcal{R}^-=\overline{\mathcal{R}^+}$. By definition we have $R_J = \mathcal{R}^{\natural} + \bar{\mathcal{R}}^{\natural} + \mathcal{R}^+ + \mathcal{R}^-$. Since \cref{eqn:DeformationComplex2} is a complex we have
\begin{equation}\label{eqn:bigraded-complex-meaning}
\mathcal{R}^{\natural}\,\mathcal{D} + \mathcal{R}^+ \bar{\mathcal{D}} = 0
\quad \text{ and }\quad
\mathcal{D}^* \,\mathcal{R}^{\natural} + \bar{\mathcal{D}}^*\, \mathcal{R}^- = 0.
\end{equation}

Next we observe that the linearized operators governing stably embeddable deformations and the CR obstruction density arise naturally from this picture. Let $(M,H,J)$ be a CR $3$-manifold. We may define $\mathcal{D}$, $\mathcal{D}^*$ and their conjugates on $M$ by the same local formulae \cref{eqn:calD-formula} and \cref{eqn:calD-adjoint-formula} as we used for the CR $3$-sphere. Given a section $E$ of $\mathbb{C}\otimes \mathcal{D}ef(M) = (T^{1,0})^*\otimes T^{0,1} \oplus (T^{0,1})^*\otimes T^{1,0}$ we write $E^{(2,0)}$ for the $(T^{1,0})^*\otimes T^{0,1}$ part and $E^{(0,2)}$ for the $(T^{0,1})^*\otimes T^{1,0}$ part, and similarly for sections of $\mathbb{C}\otimes \mathcal{C}urv(M)$. Then $(\DJ f)^{(2,0)}=\mathcal{D}f$, and $(\DJ f)^{(0,2)}=\bar{\mathcal{D}}f$.
Now if $M$ is a strictly pseudoconvex hypersurface in $\mathbb{C}^2$, then it follows from \cref{lem:variational-vector} that $E$ is the infinitesimal deformation tensor of a contact parametrized deformation of $M$ in $\mathbb{C}^2$ if and only if $E^{(0,2)}=-2\bar{\mathcal{D}}f$ for some complex density $f$. The $-2$ is chosen to match with our previous conventions. Let $E=-2(\mathcal{D}\bar{f}+\bar{\mathcal{D}}f)$ be such an infinitesimal deformation tensor. If $f$ is real, then $E=-2(\mathcal{D}f+\bar{\mathcal{D}}f)$ and $E$ is the trivial infinitesimal deformation arising from the flow of the contact Hamiltonian vector field $V_f$. If $f=iv$ is imaginary then $E=2(\mathcal{D}f-\bar{\mathcal{D}}f)$, and any contact parametrized deformation $\psi_t$ of $M\subset\mathbb{C}^2$ inducing $E$ satisfies $\left.\frac{d}{dt}\right|_{t=0} \psi_t = J V_v$, where here $J$ is the standard complex structure on $\mathbb{C}^2$.

To see how the linearized obstruction density appears in \cref{eqn:bigraded-deformation-complex}, we first observe that for a general CR $3$-manifold the CR obstruction density \cref{eqn:obstruction-density} may be written as
\begin{equation*}
\mathcal{O}=\frac{1}{3}\mathcal{D}^* Q^{(2,0)}.
\end{equation*}
Let $\psi_t:S^3\to \mathbb{C}^2$, $t\in[0,\epsilon)$, be a contact parametrized deformation of the unit sphere $S^3\subset \mathbb{C}^2$, and let $(S^3,H,J_t)$ denote the corresponding family of induced CR structures on $S^3$. Let $\mathcal{O}_t$ denote the CR obstruction density of $(S^3,H,J_t)$, and let $\dot{\mathcal{O}}=\left.\frac{d}{dt}\right|_{t=0}\mathcal{O}_t$. Then, since $Q^{(2,0)}=0$ for the unit sphere, differentiating $\mathcal{O}_t=\frac{1}{3}\mathcal{D}_t^* Q_t^{(2,0)}$ at $t=0$ gives
\begin{equation}\label{eqn:dotO-for-sphere-deformations}
\dot{\mathcal{O}} = \frac{1}{3}\mathcal{D}^* \dot{Q}^{(2,0)},
\end{equation}
where $\dot{Q}^{(2,0)}=\left.\frac{d}{dt}\right|_{t=0}Q_t^{(2,0)}$.

\subsection{First order obstruction flatness and CR flatness of deformations}\label{subsec:first-order-obstruction-flatness}
We now apply the deformation complex to prove \cref{thm:FormalDeformations} for the case where $\mathcal{O}_t$ vanishes to second order in the deformation parameter $t$, see \cref{thm:obstruction-dot-zero-implies-Q-dot-zero} below. We first observe that, for deformations of the unit sphere, $\dot{\mathcal{O}}$ and $\dot{Q}$ are independent of the choice of contact parametrization of the deformation. Given a smooth family $M_t$, $t\in[0,\epsilon)$, of compact strictly pseudoconvex hypersurfaces in $\mathbb{C}^2$, we call a contact parametrized deformation $\psi_t:M_0\to \mathbb{C}^2$ with $\psi_t(M_0)=M_t$, for all $t$, a \emph{contact parametrization} of the family $M_t$.
\begin{lemma}\label{lem:contact-parametrization-independence}
Let $M_t\subset \mathbb{C}^2$, $t\in[0,\epsilon)$, be a smooth family of compact strictly pseudoconvex hypersurfaces with $M_0$ the unit sphere. Let $\psi_t:S^3\to \mathbb{C}^2$ be a contact parametrization of this family, and let $(S^3,H,J_t)$ denote the corresponding family of CR structures on $S^3$. Let $\dot{\mathcal{O}}=\left.\frac{d}{dt}\right|_{t=0}\mathcal{O}_t$ and $\dot{Q}=\left.\frac{d}{dt}\right|_{t=0}Q_t$. Then $\dot{Q}$ and $\dot{\mathcal{O}}$ are independent of the choice of contact parametrization $\psi_t$.
\end{lemma}
\begin{proof}
The lemma is equivalent to the statement that $\dot{Q}$ (and hence also $\dot{\mathcal{O}}=\frac{1}{3}\mathcal{D}^*\dot{Q}^{(2,0)}$) is unchanged under contact reparametrization of the abstract deformation $(S^3,H,J_t)$. But this was already observed in seeing that \cref{eqn:DeformationComplex2} is a complex, and is equivalent to $\RJ\DJ$ being zero.
\end{proof}
This lemma allows us to work with a convenient choice of contact parametrization. Let $\psi_t:S^3\to \mathbb{C}^2$, $t\in[0,\epsilon)$, be a contact parametrized deformation of the unit sphere such that $f=\btheta(\dot{\psi})$ imaginary, and let $(S^3,H,J_t)$ denote the corresponding family of CR structures on $S^3$. If we let $E = \left.\frac{d}{dt}\right|_{t=0} J_t$, then $E=2(\mathcal{D}f-\bar{\mathcal{D}}f)$. Taking the $(2,0)$ part of $\dot{Q}=\RJ E$ gives
\begin{equation}\label{eqn:Q-20-dot-for-f-imaginary}
\dot{Q}^{(2,0)} = 2(\mathcal{R}^{\natural}\mathcal{D}-\mathcal{R}^+\bar{\mathcal{D}})f = -4\mathcal{R}^+\bar{\mathcal{D}}f
\end{equation}
since $\mathcal{R}^{\natural}\,\mathcal{D}=-\mathcal{R}^+\bar{\mathcal{D}}$. Thus
\begin{equation}\label{eqn:O-dot-for-f-imaginary}
\dot{\mathcal{O}} = -\frac{4}{3}\mathcal{D}^*\,\mathcal{R}^+\bar{\mathcal{D}}f.
\end{equation}
These observations lead us to the following theorem, which implies \cref{thm:FormalDeformations} for the case where $\mathcal{O}_t$ vanishes to second order in the deformation parameter $t$.
\begin{theorem}\label{thm:obstruction-dot-zero-implies-Q-dot-zero}
Let $\psi_t:S^3\to \mathbb{C}^2$, $t\in[0,\epsilon)$, be a contact parametrized deformation of the unit sphere $S^3\subset \mathbb{C}^2$, and let $(S^3,H,J_t)$ denote the corresponding family of CR structures on $S^3$. If $\dot{\mathcal{O}}=0$, then $\dot{Q}=0$.
\end{theorem}
\begin{proof}
By \cref{lem:contact-parametrization-independence} and \cref{lem:making-psi-dot-imaginary} it is no loss of generality to assume $f=\btheta(\dot{\psi})$ is imaginary. Let $E = \left.\frac{d}{dt}\right|_{t=0} J_t$, so that $E=2(\mathcal{D}f-\bar{\mathcal{D}}f)$.
By \cref{eqn:O-dot-for-f-imaginary} we then have
\begin{equation}\label{eqn:zigzag-composition-zero-on-f}
\mathcal{D}^*\,\mathcal{R}^+\bar{\mathcal{D}}f =0.
\end{equation}
Let $(\theta,\theta^1,\theta^{\oneb})$ be the standard admissible coframe on $S^3$ (for which $A^1{}_{\oneb}=0$ and $R=2$). As usual we use $\theta$ to trivialize $\cE(1,1)$ and think of $f$ as a function. We define $\varphi_{\oneb}{}^1$ to be $-if^1{}_{\oneb}$ where $f^1{}_{\oneb}=\nabla_{\oneb}\nabla^1 f$, so that
\begin{equation*}
E = 2i\varphi_{1}{}^{\oneb}\theta^{1}\otimes Z_{\oneb} - 2i\varphi_{\oneb}{}^1\theta^{\oneb}\otimes Z_1.
\end{equation*}
Continuing to denote derivatives of $f$ by appending indices, by \cref{eqn:calR-plus} and \cref{eqn:Q-20-dot-for-f-imaginary} we have
\begin{equation}\label{eqn:Q-20-dot-for-proof}
\dot{Q}^{(2,0)} = -4\mathcal{R}^+ \bar{\mathcal{D}}f = -2\mathcal{R}^+(2i\varphi_{\oneb}{}^1\theta^{\oneb}\otimes Z_1 )
= \frac{1}{3} f_{\oneb\oneb}{}^{\oneb\oneb}{}_{1}{}^{\oneb} \theta^{1}\otimes Z_{\oneb}.
\end{equation}
Applying $\bar{\mathcal{D}}^*$ to the above, by \cref{eqn:zigzag-composition-zero-on-f} (or \cref{eqn:dotO-for-sphere-deformations})
\begin{equation*}
f_{\oneb\oneb}{}^{\oneb\oneb}{}_{11}{}^{11}=0.
\end{equation*}
Integrating by parts twice we have
\begin{equation*}
0=\int_M f^{11}{}_{11} \,\overline{f_{\oneb\oneb}{}^{\oneb\oneb}{}_{11}{}^{11}}
= \int_M f^{11}{}_{11}{}^{\oneb\oneb} \, \overline{f_{\oneb\oneb}{}^{\oneb\oneb}{}_{11}}
= \int_M |f_{\oneb\oneb}{}^{\oneb\oneb}{}_{11}|^2,
\end{equation*}
where the integrals are taken with respect to $\theta\wedge\dee\theta$. So $f_{\oneb\oneb}{}^{\oneb\oneb}{}_{11}=0$, and \cref{eqn:Q-20-dot-for-proof} gives $\dot{Q}^{(2,0)}=0$. Since $\dot{Q}$ is real, $\dot{Q}=2\mathrm{Re}\,\dot{Q}^{(2,0)} =0$, as required.
\end{proof}

\subsection{Higher order obstruction flatness and CR flatness of deformations}\label{subsec:higher-order-obstruction-flatness}
Our aim is now to prove the following theorem, which implies \cref{thm:FormalDeformations}.
\begin{theorem}\label{thm:FormalDeformations-v2}
Let $\psi_t:S^3\to \mathbb{C}^2$, $t\in[0,\epsilon)$, be a smooth contact parametrized deformation of the unit sphere $S^3\subset \mathbb{C}^2$, and let $(S^3,H,J_t)$ denote the corresponding family of CR structures on $S^3$. If $\mathcal{O}_t=O(t^k)$, then $Q_t=O(t^k)$.
\end{theorem}

The proof of this theorem requires several results, which straightforwardly generalize the above results for the case of first order deformations. First of all we show that the conditions $\mathcal{O}_t=O(t^k)$ and $Q_t=O(t^k)$ are geometric conditions on the deformation, that is these conditions only depend on the family of strictly pseudoconvex hypersurfaces $M_t\subset \mathbb{C}^2$ (with $M_0$ the unit sphere), and not on the contact parametrization of $M_t$.
\begin{lemma}\label{lem:geometric-invariance-of-higher-order-vanishing}
Let $\psi_t:S^3\to \mathbb{C}^2$, $t\in[0,\epsilon)$, be a smooth contact parametrized deformation of the unit sphere $S^3\subset \mathbb{C}^2$, and let $(S^3,H,J_t)$ denote the corresponding family of CR structures on $S^3$. Suppose $Q_t=O(t^{\ell})$ and $\mathcal{O}_t=O(t^k)$. If $\psi'_t$ is any contact reparametrization of $\psi_t$, inducing the family of CR structures $(S^3,H,J'_t)$ on $S^3$, then $Q'_t =O(t^{\ell})$ and $\mathcal{O}'_t=O(t^k)$.
\end{lemma}
\begin{proof}
Let $\psi'_t$ be given by $\psi_t\circ \varphi_t$, where $\varphi_t$ is a smooth family of contact diffeomorphisms of $(S^3,H)$. Then $J'_t = \varphi_t^* J_t$, so by the naturality of the Cartan umbilical tensor we have $Q'_t = \varphi_t^* Q_t$. Since $Q_t=\frac{1}{\ell !}t^{\ell}Q^{(\ell)}+\frac{1}{(\ell+1)!}t^{\ell+1}Q^{(\ell+1)}+O(t^{\ell+2})$ we have
\begin{equation*}
Q'_t = \varphi_t^* \left(\frac{1}{\ell !}t^{\ell}Q^{(\ell)} + O(t^{\ell+1})\right) = \frac{1}{\ell !}t^{\ell}\varphi_t^*Q^{(\ell)} + O(t^{\ell+1}) = O(t^{\ell}),
\end{equation*}
as required. The argument for the CR obstruction density is the same.
\end{proof}
We shall also need the fact that these two conditions are biholomorphically invariant, that is, invariant under composition with a smooth family of ambient biholomorphisms.
\begin{lemma}\label{lem:biholomorphic-invariance-of-order-of-vanishing}
Let $\psi_t:S^3\to \mathbb{C}^2$, $t\in[0,\epsilon)$, be a smooth contact parametrized deformation of the unit sphere $S^3\subset \mathbb{C}^2$, and let $(S^3,H,J_t)$ denote the corresponding family of CR structures on $S^3$. Suppose $Q_t=O(t^{\ell})$ and $\mathcal{O}_t=O(t^k)$. Let $U_t$ be a neighborhood of $\psi_t(S^3)\subset \mathbb{C}^2$ and $\Phi_t:U_t\to\mathbb{C}^2$, $t\in[0,\epsilon)$, a smooth family of biholomorphisms with $\Phi_0=\mathrm{id}$. If $\psi'_t=\Phi_t\circ\psi_t$ induces the family of CR structures $(S^3,H,J'_t)$ on $S^3$, then $Q'_t=O(t^{\ell})$ and $\mathcal{O}'_t=O(t^k)$.
\end{lemma}
\begin{proof}
The result is immediate from the fact that the Cartan umbilical tensor and CR obstruction density are weighted biholomorphic invariants (cf.\ \cite{Fefferman1979, Graham1987b}), so that we have $Q'_t = |\det \Phi_t|^{-4/3}Q_t$ and $\mathcal{O}'_t = |\det \Phi_t|^{-2}\mathcal{O}_t$.
\end{proof}

\begin{remark}\label{rem:geometric-invariance}
Clearly, the conclusion of \cref{lem:biholomorphic-invariance-of-order-of-vanishing} holds even if the open subsets $U_t$ only contain $M_t:=\psi_t(S^3)$ in their closures and the biholomorphisms $\Phi_t$ extend smoothly to $M_t$. By Hartog's theorem the same is true if $\Phi_t$ is replaced by a smooth family of CR embeddings $\phi_t:M_t\to\mathbb C^2$.
\end{remark}

In proving \cref{thm:FormalDeformations-v2} our strategy will be to repeatedly normalize the contact parametrized deformation by contact reparametrizations and by ambient biholomorphisms (or CR embeddings), allowing us to argue as in the proof of \cref{thm:obstruction-dot-zero-implies-Q-dot-zero}. For the normalization procedure we need several lemmas. The first is the following straightforward generalization of \cref{lem:contact-reparametrization-Jdot}.
\begin{lemma}\label{lem:deformation-trivial-at-order-k}
Let $(M,H,J_t)$, $t\in[0,\epsilon)$, be a smooth family of CR structures on $M$. If
\begin{equation*}
J_t = J_0 + \frac{1}{k!}t^k E + O(t^{k+1})
\end{equation*}
where $E$ is a trivial infinitesimal deformation tensor, then there exists a smooth family $\varphi_t$, $t\in[0,\epsilon)$, of contact diffeomorphisms of $(M,H)$ such that
\begin{equation*}
\varphi_t^*J_t = J_0 + O(t^{k+1}).
\end{equation*}
\end{lemma}
We will also need the following lemma.
\begin{lemma}\label{lem:ambient-CR-normalization}
Let $\psi_t:M\to \mathbb{C}^2$, $t\in[0,\epsilon)$, be a contact parametrized deformation of the compact strictly pseudoconvex hypersurface $M\subset \mathbb{C}^2$ with $\psi_t = \psi_0 + O(t^k)$, and let $(M,H,J_t)$ be the corresponding family of CR structures on $M$. Suppose $J_t = J + O(t^{k+1})$. Then there exists a smooth family of CR embeddings $\phi_t:\psi_t(M)\to\mathbb{C}^2$, $t\in[0,\epsilon)$, with $\phi_0=\mathrm{id}$ such that $\phi_t\circ \psi_t = \psi_0 + O(t^{k+1})$.
\end{lemma}
To state the next lemma we need a definition. Let $\psi_t:M\to \mathbb{C}^2$, $t\in[0,\epsilon)$, be a contact parametrized deformation of the strictly pseudoconvex hypersurface $M\subset \mathbb{C}^2$, and let $(M,H,J_t)$ denote the corresponding family of CR structures on $M$. If $\psi_t = \psi_0 + \frac{1}{k!}t^k\xi + O(t^{k+1})$ then we may think of $\xi=\left.\frac{d^k}{dt^k}\right|_{t=0}\psi_t$ as a section of $T\mathbb{C}^2|_M$, and we define $\psi^{(k)}$ to be the corresponding section of $T_{(1,0)}$. When $k=1$ we have $\psi^{(1)}=\dot{\psi}$. The following lemma is a straightforward generalization of \cref{lem:making-psi-dot-imaginary}.
\begin{lemma}\label{lem:making-psi-k-imaginary}
Let $\psi_t:M\to \mathbb{C}^2$, $t\in[0,\epsilon)$, be a contact parametrized deformation of the compact strictly pseudoconvex hypersurface $M\subset \mathbb{C}^2$ with $\psi_t = \psi_0 + O(t^k)$. Then there is a smooth family $\varphi_t$, $t\in[0,\epsilon)$, of contact diffeomorphisms of $M$ which reparametrizes $\psi_t$ to $\psi'_t = \psi_t\circ \varphi_t$ with $\btheta(\psi'^{(k)})$ imaginary.
\end{lemma}
Finally, in order to run the argument we used in the proof of \cref{thm:obstruction-dot-zero-implies-Q-dot-zero} we need another straightforward generalization of our results from the case of first order deformations:
\begin{lemma}\label{lem:formulae-for-deformations-trivial-to-higher-order}
Let $\psi_t:S^3\to \mathbb{C}^2$, $t\in[0,\epsilon)$, be a smooth contact parametrized deformation of the unit sphere $S^3\subset \mathbb{C}^2$, and let $(S^3,H,J_t)$ denote the corresponding family of CR structures on $S^3$. If $\psi_t = \psi_0 + O(t^k)$ and $f=\btheta(\psi^{(k)})$, then
\begin{equation}
J_t = J + \frac{1}{k!}t^k E + O(t^{k+1})
\end{equation}
where $E = -2(\mathcal{D}\bar{f}+\bar{\mathcal{D}}f)$. Moreover,
\begin{equation}\label{eqn:Q-first-term-higher-order-in-t}
Q_t = \frac{1}{k!}t^k \RJ E +  O(t^{k+1}),
\end{equation}
and
\begin{equation}
\mathcal{O}_t = \frac{1}{3k!}t^k \mathcal{D}^*(\RJ E)^{(2,0)}  +  O(t^{k+1}).
\end{equation}
\end{lemma}
Combining these lemmas we may now prove \cref{thm:FormalDeformations-v2}.
\begin{proof}[Proof of \cref{thm:FormalDeformations-v2}]
We note first that the case $k=1$ is trivial, and $k=2$ is the content of \cref{thm:obstruction-dot-zero-implies-Q-dot-zero}.  Now, fix a $k\geq 3$. \cref{thm:obstruction-dot-zero-implies-Q-dot-zero} implies $\dot{Q}=\RJ E = 0$, where $E = \dot J= -2(\mathcal{D}\bar{f}+\bar{\mathcal{D}}f)$ and $f=\btheta(\dot \psi)$. Since \cref{eqn:DeformationComplex2} is exact at $\mathcal{D}ef(S^3)$, it follows that $E$ is a trivial infinitesimal deformation tensor. By \cref{lem:deformation-trivial-at-order-k} (or \cref{lem:contact-reparametrization-Jdot}), contact reparametrizing $\psi_t$ if necessary, we may assume $E=0$ (i.e. $J_t = J+O(t^2)$). By \cref{lem:ambient-CR-normalization}, composing $\psi_t$ with a smooth family of CR embeddings if necessary, we may assume that $\psi_t=\psi_0+O(t^2)$. By \cref{lem:making-psi-k-imaginary}, further contact reparametrizing $\psi_t$ if necessary, we may assume that $\btheta(\psi^{(2)})$ is imaginary. By \cref{lem:geometric-invariance-of-higher-order-vanishing,lem:biholomorphic-invariance-of-order-of-vanishing} (cf.\ \cref{rem:geometric-invariance}) the conditions $Q_t=O(t^{2})$ and $\mathcal{O}_t=O(t^k)$ are preserved under these normalizations on $\psi_t$. These observations form the base case of an induction. Let $\ell\in \{2,\ldots,k-1\}$ and suppose that $Q_t=O(t^{\ell})$ and that, after contact reparametrizing and composing $\psi_t$ with a smooth family of CR embeddings if necessary, $\psi_t=\psi_0+O(t^{\ell})$ and $\btheta(\psi^{(\ell)})$ is imaginary. Since $f=\btheta(\psi^{(\ell)})$ is imaginary, by \cref{lem:formulae-for-deformations-trivial-to-higher-order} we have $J_t=J+\frac{1}{\ell !}t^\ell E^{(\ell)} + O(t^{\ell+1})$ with $E^{(\ell)} = 2(\mathcal{D}f-\bar{\mathcal{D}}f)$. But then $\mathcal{O}_t = \frac{1}{3\ell !}t^{\ell} \mathcal{D}^*(\RJ E)^{(2,0)}  +  O(t^{\ell+1})$, where
\begin{equation*}
\mathcal{D}^*(\RJ E)^{(2,0)} = \mathcal{D}^*(2\mathcal{R}^{\natural}\,\mathcal{D}f-2\mathcal{R}^+ \bar{\mathcal{D}}f)
= -4\mathcal{D}^*\mathcal{R}^+ \bar{\mathcal{D}}f.
\end{equation*}
Computing as in the proof of \cref{thm:obstruction-dot-zero-implies-Q-dot-zero}, with respect to the standard admissible coframe on the CR $3$-sphere we obtain $\mathcal{D}^*\mathcal{R}^+ \bar{\mathcal{D}}f = \frac{1}{3}f_{\oneb\oneb}{}^{\oneb\oneb}{}_{11}{}^{11}$, where the indices denote covariant derivatives. But since $\ell<k$, we again have $f_{\oneb\oneb}{}^{\oneb\oneb}{}_{11}{}^{11}=0$. By the same argument as in the proof of \cref{thm:obstruction-dot-zero-implies-Q-dot-zero} it follows that $f_{\oneb\oneb}{}^{\oneb\oneb}{}_{11}=0$. But $(\RJ E)^{(2,0)}= \frac{1}{3}f_{\oneb\oneb}{}^{\oneb\oneb}{}_{1}{}^{\oneb}\theta^1\otimes Z_{\oneb}$, so $\RJ E=0$ and by \cref{eqn:Q-first-term-higher-order-in-t} we have $Q_t=O(t^{\ell+1})$. In order to be able to continue the induction we observe $\RJ E=0$ implies $E$ is a trivial infinitesimal deformation tensor so that by \cref{lem:deformation-trivial-at-order-k,lem:ambient-CR-normalization,lem:making-psi-k-imaginary}, after contact reparametrizing and composing $\psi_t$ with a smooth family of CR embeddings if necessary, we obtain $\psi_t=\psi_0+O(t^{\ell+1})$ with $\btheta(\psi^{(\ell+1)})$ imaginary. By \cref{lem:geometric-invariance-of-higher-order-vanishing,lem:biholomorphic-invariance-of-order-of-vanishing} (again cf.\ \cref{rem:geometric-invariance}) the conditions $Q_t=O(t^{\ell+1})$ and $\mathcal{O}_t=O(t^k)$ are preserved under these normalizations on $\psi_t$. The result follows.
\end{proof}

\subsection{An Example: Real Ellipsoids in \texorpdfstring{$\mathbb{C}^2$}{C2}} \label{sec:real-ellipsoids}
Let $S^3$ be the standard CR $3$-sphere. Consider the family of ellipsoids defined by $r_t=0$, where
$$
r_t:= 1 - (|z|^2 + |w|^2) - t \left(A (\mathrm{Re}\,z)^2 + B (\mathrm{Re}\,w)^2 \right), \quad t\geq 0.
$$
In \cite{EbenfeltZaitsev-arxiv2016} it is shown that, after pulling back the corresponding CR structures to $S^3$ one has
$$
Q_t=\frac{1}{2}t^2{Q^{(2)}} +O(t^3),
$$
where $Q^{(2)}\not\equiv 0$ on the sphere $S^3$. Thus, $\dot Q=0$ and therefore  \cref{thm:obstruction-dot-zero-implies-Q-dot-zero} does not give us any information about the order of vanishing (or nonvanishing) of the obstruction. However, by \cref{thm:FormalDeformations-v2}, we conclude that there is $\mathcal O^{(2)}$ not identically zero on $S^3$ such that
$$
\mathcal O_t=\frac{1}{2}t^2{\mathcal{O}^{(2)}} + O(t^3).
$$

\section{The CR tractor calculus}\label{sec:TractorCalculus}

Canonically associated with any CR $3$-manifold $(M,H,J)$ is a (torsion free, normal) Cartan geometry of type $(\mathrm{PU}(2,1),K)$, where the projective unitary group $\mathrm{PU}(2,1)$ acts on the CR $3$-sphere by fractional linear transformations, and $K$ is the stabilizer subgroup of a point, so that $S^3=\mathrm{PU}(2,1)/K$. Moreover, this correspondence induces in a natural way an equivalence of categories \cite{CapSlovak2009}. In particular, this means that infinitesimal symmetries and deformations of CR $3$-manifolds may equivalently be described in terms of infinitesimal symmetries and deformations of the corresponding Cartan geometry. The results of taking this point of view are described in detail for the more general case of parabolic geometries in \cite{Cap2008}. Below we recall in our specific setting a result of \v{C}ap \cite{Cap2008} on the Cartan geometric descripton of infinitesimal symmetries, and use this result to prove \cref{thm:Symmetries}. The result of \cite{Cap2008} we require is formulated in terms of \emph{tractor calculus} \cite{CapGover2002}. The tractor calculus on a (parabolic) Cartan geometry modelled on $G/P$ is the calculus of associated vector bundles induced by representations of $G$ (the so-called \emph{tractor bundles}). We recall below the necessary background on the tractor calculus of CR $3$-manifolds. Rather than first constructing the CR Cartan connection, our approach is to directly construct the standard tractor bundle and connection, specializing the treatment of \cite{GoverGraham2005} to the $3$-dimensional case. The Cartan bundle and connection may be readily recovered from this by passing to an adapted frame bundle. This gives a direct and highly practical approach to the Cartan geometry of CR $3$-manifolds.

As is usual in CR geometry, it will be convenient for us to work with the group $G=\mathrm{SU}(2,1)$, and the stabilizer subgroup $P$ giving $S^3=G/P$, rather than $(\mathrm{PU}(2,1),K)$. The Cartan geometry of type $(\mathrm{PU}(2,1),K)$ corresponding to $(M,H,J)$ may be lifted to a Cartan geometry of type $(G,P)$ if and only if the holomorphic tangent bundle (or equivalently the canonical bundle) admits a cube root (see, e.g., \cite{CapSlovak2009}). For any CR $3$-manifold such a lift always exist locally, and globally for strictly pseudoconvex hypersurfaces in $\mathbb{C}^2$. We will always assume that $(M,H,J)$ admits such a lift, i.e. that the integral first Chern class of $T^{1,0}$ is divisible by $3$. This is also a necessary condition for the global existence of the standard tractor bundle, since this is induced by a $G$-representation that does not descend to a $\mathrm{PU}(2,1)$-representation. The \emph{adjoint tractor bundle}, on the other hand, which is induced by the adjoint representation of $G$, is always globally well defined. These bundles are discussed in detail below.

\subsection{CR densities and holomorphic tangent vectors}\label{subsec:CR-densities}
Let $(M,H,J)$ be a CR $3$-manifold, and let $\Lambda^{1,0}$ denote the complex rank $2$ bundle of $(1,0)$-forms on $M$. The bundle $\Lambda^{2,0}=\Lambda^2(\Lambda^{1,0})$ of $(2,0)$-forms is referred to as the \emph{canonical line bundle} of $M$, and denoted by $\scrK$. We assume throughout that its dual $\scrK^*$ admits a (global) cube root, which we fix and denote by $\cE(1,0)$. (If $M$ is a strictly pseudoconvex hypersurface in $\mathbb{CP}^2$, then we may take the complex line bundle $\cE(1,0)$ to be $\mathcal{O}(1)|_M$ where $\mathcal{O}(1)$ is the hyperplane bundle on $\mathbb{CP}^2$.) We then define the \emph{CR density line bundle} of weight $(w,w')$ to be $\cE(w,w') = \cE(1,0)^w \otimes \overline{\cE(1,0)}^{w'}$, where $w,w'\in \mathbb{C}$ with $w-w'\in \mathbb{Z}$. Note that for $w$ real the bundle $\cE(w,w)$ is invariant under conjugation, and hence contains a real subbundle $\cE_{\mathbb{R}}(w,w)$. The CR density bundles exhaust the so called natural line bundles on CR manifolds \cite{CapSlovak2009}, the upshot of which is that we will be able to naturally identify all the more familiar line bundles on $M$ with one of these bundles. Though we will not do away with the usual bundles, it will be useful to record their weights as CR density bundles. Note that by definition $\cE(3,0)=\scrK^*$, so $\cE(-3,0)=\scrK$.

Trivializing the bundle $TM/H$ determines a contact form on $M$ via the natural map $TM\rightarrow TM/H$.
Similarly, a choice of non-vanishing section $\zeta$ (i.e.\ a trivialization) of $\scrK$ determines canonically a contact form $\theta$ on $M$ by the requirement \cite{Farris1986} (see also \cite{Lee1988}) that
\begin{equation}\label{eq:volume-normalized}
\theta\wedge \dee\theta = i\theta \wedge (T\intprod \zeta)\wedge (T\intprod\overline{\zeta}).
\end{equation}
In this case we say that $\theta$ is \emph{volume normalized} with respect to $\zeta$. Combining these observations, we may realize $TM/H$ as a real CR density line bundle (cf.\ \cref{sec:pseudohermitian}) as follows. A contact form $\theta$ determines canonically a section $|\zeta|^2=\zeta\otimes\overline{\zeta}$ of $\scrK\otimes\overline{\scrK} = \cE(-3,-3)$ by the condition that $\zeta$ satisfy \cref{eq:volume-normalized} ($\zeta$ is only determined up to phase at each point). If we rescale $\theta$ to $\thetah=e^{\Ups}\theta$, with $\Ups\in C^{\infty}(M,\mathbb{R})$, then the corresponding section $|\hat{\zeta}|^2$ equals $e^{3\Ups}|\zeta|^2$. Thus, the map which assigns to a contact form $\theta$ the section $|\zeta|^{2/3}$ of $\cE_{\mathbb{R}}(-1,-1)$ extends to a canonical isomorphism of $H^{\perp}$ with $\cE_{\mathbb{R}}(-1,-1)$. Dually $TM/H$ is canonically isomorphic to $\cE_{\mathbb{R}}(1,1)$, explaining the notation we used earlier. Recall that this identification gives us a tautological $1$-form $\btheta$ of weight $(1,1)$, corresponding to the map $TM\to TM/H=\cE_{\mathbb{R}}(1,1)$.

We define the \emph{CR Levi form} $\Levi:T^{1,0}\otimes T^{0,1} \to \mathbb{C}TM/\mathbb{C}H$ by
\begin{equation*}
\Levi(U,\overline{V}) = 2i[U,\overline{V}]\;\, \mathrm{mod}\; \mathbb{C}H.
\end{equation*}
On a strictly pseudoconvex CR $3$-manifold the CR Levi form is a bundle isomorphism, so we have $T^{1,0}\otimes T^{0,1} \cong \mathbb{C}TM/\mathbb{C}H = \cE(1,1)$. The CR Levi form may be interpreted as a Hermitian bundle metric on $T^{1,0}\otimes\cE(-1,0)$, and we would like to have a more concise notation for bundles like this one. We use the symbol $\cE$ decorated with appropriate indices to denote the tensor bundles constructed from $T^{1,0}$ and $T^{0,1}$. For example, $\cE^1=T^{1,0}$, $\mathcal{E}_{\oneb}=(T^{0,1})^*$, and $\mathcal{E}_{1\oneb}=(T^{1,0})^*\otimes(T^{0,1})^*$. We will now generally use abstract index notation for sections of these bundles. So, for example, $V^1$ may denote a global section of $\cE^1=T^{1,0}$ (previously written locally as $V^1Z_1$). This keeps the notation from getting too heavy, and allows us to globalize our previous local formulas. Generally we denote the tensor product of a complex vector bundle $\mathcal{V}$ on $M$ with $\cE(w,w')$ by appending $(w,w')$, as in $\mathcal{V}(w,w')$. The CR Levi form will be thought of as a section $\bh_{1\oneb}$ of $\cE_{1\oneb}(1,1)$, with inverse $\bh^{1\oneb}$. The Levi form will be used to identify $\cE^{1\oneb}$ with $\cE(1,1)$, and $\cE_{1\oneb}$ with $\cE(-1,-1)$, and to raise and lower indices without comment.

From the general theory \cite{CapSlovak2009} we know that the bundle $\cE^1=T^{1,0}$ can be identified with a density bundle of some weight. Since $\cE^{1\oneb}=\cE(1,1)$ we see that $\cE^1=\cE(w,1-w)$ for some $w$. Recalling that $\Lambda^3$ may be canonically identified with $\cE_{\mathbb{R}}(-2,-2)$, the exact weight can be determined by noting that the wedge product gives a canonical identification $\scrK\otimes (T^{0,1})^* = \mathbb{C}\Lambda^3$, i.e. $\cE(-3,0)\otimes \cE_{\oneb} = \cE(-2,-2)$, so that $\cE_{\oneb} = \cE(1,-2)$ and hence $\cE_1=\cE(-2,1)$ and $\cE^1=\cE(2,-1)$.

\subsection{Weighted pseudohermitian calculus}

In order to construct the standard tractor bundle and connection directly from the standard pseudohermitian calculus, we first observe that the Tanaka-Webster connection $\nabla$ of a pseudohermitian structure $\theta$ extends naturally to act on the CR density bundles, since $\nabla$ acts on the canonical bundle $\scrK$. Since the Tanaka-Webster connection of $\theta$ preserves $\theta$, and also preserves the section $|\zeta|^2$ of $\scrK\otimes\overline{\scrK} = \cE(-3,-3)$ determined by volume normalization, the Tanaka-Webster connection respects the CR invariant identification of $TM/H$ with $\cE_{\mathbb{R}}(1,1)$. Another way of saying this is that $\nabla \btheta = 0$. A similar argument shows that $\nabla$ preserves the CR Levi form, $\nabla \Levi =0$. Hence, the Tanaka-Webster connection of $\theta$ respects all of the CR invariant identifications made in \cref{subsec:CR-densities}. We therefore make use of CR densities whenever convenient.

Given a choice of admissible coframe $(\theta,\theta^1,\theta^{\oneb})$ we now take components of tensors with respect to $(\btheta,\theta^1,\theta^{\oneb})$. This means that if $V$ is a tangent vector, then $V^0$ has weight $(1,1)$, and is globally well defined (and independent of $\theta$). A choice of global contact form allows us to decompose the complexified tangent bundle $\mathbb{C}TM$ as $\cE^1\oplus\cE^{\oneb}\oplus\cE(1,1)$. Using abstract index notation we may therefore decompose $V$ globally as $V\overset{\theta}{=} (V^1,V^{\oneb},V^0)$. If $(\theta,\theta^1,\theta^{\oneb})$ and $\thetah=e^{\Ups}\theta$, then writing $\thetah^1=\theta^1+i\Ups^1\btheta$ where $\Ups^1=\nabla^1\Ups$ it is easy to see that $(\thetah, \thetah^1,\thetah^{\oneb})$ is again an admissible coframe. It follows that if $V\overset{\theta}=(V^1,V^{\oneb},V^0)$, then
\begin{equation}
V\overset{\thetah}{=} (V^1+iV^0\Ups^1,V^{\oneb}-iV^0\Ups^{\oneb},V^0).
\end{equation}
Dually, for a $1$-form $\eta$ with $\eta\overset{\theta}{=}(\eta_1,\eta_{\oneb},\eta_0)$ we have
\begin{equation}\label{1-FormTransformation}
\eta \overset{\thetah}{=} (\eta_1,\eta_{\oneb},\eta_0 - i\Ups^1\eta_1 +i\Ups^{\oneb}\eta_{\oneb}).
\end{equation}

We will need to commute derivatives of weighted tensor fields, for this we need to know the curvature of the CR density bundles. Let $\tau$ be a section of $\cE(w,w')$. From \cref{eqn:Ricci-identity,eqn:Tanaka-Webster-commuting-10-derivatives} one easily obtains that
\begin{align}
\label{eqn:commuting-11bar-derivatives-on-densities}
\nabla_1\nabla_{\oneb} \tau -\nabla_{\oneb}\nabla_1 \tau + i \bh_{1\oneb}\nabla_0 \tau & = \frac{w-w'}{3}R \bh_{1\oneb}\, \tau \,; \\
\label{eqn:commuting-10-derivatives-on-densities}
\nabla_1\nabla_{0} \tau -\nabla_{0}\nabla_1 \tau - A^{\oneb}{_{1}}\nabla_{\oneb}\tau & = \frac{w-w'}{3}(\nabla_{\oneb} A^{\oneb}{}_1) \tau,
\end{align}
cf.\ \cite[Proposition 2.2]{GoverGraham2005}. These formulae can be interpreted globally, using the abstract index formalism.

In order to directly check the CR invariance of the CR tractor connection, expressed with respect to a pseudohermitian structure, we need the following transformation laws for the Tanaka-Webster connection. If $\tau$ is a section of $\cE(w,w)$ and $\thetah=e^{\Ups}\theta$ then \cite[Proposition 2.3]{GoverGraham2005} 
\begin{align}
\label{eqn:TW1-transform-densities}\nablah_{1}\tau & =\nabla_{1}\tau +w\Ups_{1}\tau \,; \\
\label{eqn:TW1bar-transform-densities}\nablah_{\oneb}\tau & =\nabla_{\oneb}\tau+w'\Ups_{\oneb}\tau \,;\\
\label{eqn:TW0-transform-densities}\nablah_{0}\tau & =\nabla_{0}\tau-i\Ups^{1}\nabla_{1}\tau+i\Ups^{\oneb}\nabla_{\oneb}\tau\\
\nonumber & \quad+\tfrac{1}{3}\left[(w+w')\Ups_{0}+iw\Ups^{1}{_{1}}-iw'\Ups^{\oneb}{_{\oneb}}+i(w'-w)\Ups^{1}\Ups_{1}\right]\tau
\end{align}
where indices attached to $\Ups$ denote covariant derivatives, so, e.g., $\Ups_{1}{}^{1}=\nabla^1\nabla_1\Ups$. Note that in \cref{eqn:TW0-transform-densities} one of $\Ups_0$, $\Ups^{1}{_{1}}$, $\Ups^{\oneb}{_{\oneb}}$ can be eliminated by using that $\Ups_0=i(\Ups^1{_1}-\Ups^{\oneb}{_{\oneb}})$. Now either by direct calculation, or by noting that $\cE^1=\cE(2,-1)$ and using the above, we have
\begin{align}
\label{eqn:TW1-V1-transform} \nablah_{1}V^1 & =\nabla_{1}V^1+2\Upsilon_{1}V^1 \,;\\
\label{eqn:TW1bar-V1-transform} \nablah_{\oneb}V^1 & =\nabla_{\oneb}V^1 - \Upsilon_{\oneb}V^1 \,;\\
\label{eqn:TW0-V1-transform} \nablah_{0}V^1 & =\nabla_{0}V^1 -i\Upsilon^{1}\nabla_{1}V^1+i\Upsilon^{\oneb}\nabla_{\oneb}V^1+i(\Upsilon^{1}{_{1}}-\Upsilon^{1}\Upsilon_{1})V^1
\end{align}
for any section $V^1$ of $\cE^1$.
We also need the results of \cite[Lemma 2.4]{Lee1988} that if $\thetah=e^{\Ups}\theta$ then
\begin{align}
\hat{R} &= R - 2(\Ups_{1}{}^1+\Ups^1{}_{1} + \Ups^1\Ups_1) \,; \\
\hat{A}_{11} &= A_{11} + \Ups_{11} - \Ups_1\Ups_1
\end{align}
where $R$ and $A_{11}$ are interpreted as densities of respective weights $(-1,-1)$ and $(1,1)$.
We also to introduce the following higher order curvature quantities, which arise as components of the Ricci tensor of the Fefferman metric \cite{Lee1988,GoverGraham2005},
\begin{equation}\label{eqn:T_1}
T_1 = \frac{1}{12}(\nabla_1 R - 4i\nabla^1A_{11}),
\end{equation}
a section of $\cE_1(-1,-1)$, and the real $(-2,-2)$ density
\begin{equation}\label{eqn:S}
S = -(\nabla^1T_1 + \nabla^{\oneb}T_{\oneb} + \frac{1}{16}R^2 - A^{11}A_{11}).
\end{equation}
The reason these arise in the formula for the CR tractor connection below can be seen from the intimate relation between the CR tractor connection and the conformal tractor connection of the Fefferman metric \cite{CapGover2008} (cf.\ \cite{NurowskiSparling2003}). These higher order curvature quantities transform according to \cite{GoverGraham2005}
\begin{align*}
\hat{T}_1 &= T_1 +\frac{i}{2}\Ups_{01} + \frac{1}{4}R\Ups_1 - iA_{11}\Ups^1 + \frac{1}{2}\Ups_{11}\Ups^1 - \frac{1}{2}\Ups_{1\oneb}\Ups^{\oneb} - \frac{1}{2}(\Ups_1)^2\,\Ups^1\\
\hat{S} &=
S + \frac{1}{2}\Ups_{00} - 3(\Ups^1T_1+\Ups^{\oneb}T_{\oneb}) +i(\Ups_{0\oneb}\Ups^{\oneb} - \Ups_{01}\Ups^1)  + \frac{3i}{2}(A_{11}\Ups^1\Ups^1 - A_{\oneb\oneb}\Ups^{\oneb} \Ups^{\oneb})\\
& \phantom{=} -\frac{1}{4}(\Ups_0)^2  - \frac{3}{4}R\Ups_1\Ups^1 -\frac{1}{2}(\Ups_{11}\Ups^1\Ups^1 + \Ups_{\oneb\oneb}\Ups^{\oneb}\Ups^{\oneb}) + \frac{1}{2}(\Ups_{1\oneb}+\Ups_{\oneb 1})\Ups^1\Ups^{\oneb} + \frac{3}{4}(\Ups_1\Ups^1)^2.
\end{align*}

\subsection{Tractor calculus}

Let $\mathbb{C}^{2,1}$ denote the defining representation of $\mathrm{SU}(2,1)$. Let $P$ be the subgroup of $G=\mathrm{SU}(2,1)$ stabilizing a fixed isotropic line $\ell$ in $\mathbb{C}^{2,1}$. Let $(M,H,J)$ be a CR $3$-manifold and let $(\mathcal{G}\to M,\omega)$ be the canonical Cartan geometry of type $(G,P)$ corresponding to the CR structure on $M$. If $\mathbb{V}$ is an irreducible representation of $\mathrm{SU}(2,1)$ then the bundle $\mathcal{V}=\mathcal{G}\times_P \mathbb{V}$ is called a \emph{CR tractor bundle}. Every irreducible representation $\mathbb{V}$ of $\mathrm{SU}(2,1)$ is contained in some tensor representation constructed from $\mathbb{C}^{2,1}$ and $(\mathbb{C}^{2,1})^*$ as a subspace of tensors satisfying certain symmetries and the trace-free condition. It follows that knowledge of the so called \emph{(CR) standard tractor bundle} $\mathcal{T}=\mathcal{G}\times_P \mathbb{C}^{2,1}$ is sufficient to recover all of the tractor bundles. The standard tractor bundle $\mathcal{T}\to M$ should be thought of as a $P$-vector bundle, which is equivalent to saying that it is canonically equipped with a signature $(2,1)$ Hermitian bundle metric (since $P\subset \mathrm{SU}(2,1)$) and that the fibers of $\mathcal{T}$ are canonically filtered vector spaces
\begin{equation*}
\mathcal{T}_x^1 \subset \mathcal{T}_x^0 \subset  \mathcal{T}_x, \qquad x\in M 
\end{equation*}
where $\mathcal{T}_x^1$ is an isotropic line and $\mathcal{T}_x^0 = (\mathcal{T}_x^1)^{\perp}$ (since $P$ preserves the filtration $\ell \subset \ell^{\perp} \subset \mathbb{C}^{2,1}$). The $P$-principal Cartan bundle $\mathcal{G}\to M$ may readily be recovered from the standard tractor bundle as the bundle of $P$-adapted frames, that is, frames where the first frame vector is chosen from $\mathcal{T}^1$, the second from $\mathcal{T}^0$, and the frame is normalized so that the signature $(2,1)$ bundle metric takes the form
\begin{equation*}
\left(\begin{array}{ccc}
0 & 0 & 1\\
0 & 1 & 0\\
1 & 0 & 0
\end{array}\right).
\end{equation*}
Moreover, the canonical CR Cartan connection $\omega$ on $\mathcal{G}\to M$ may equivalently be viewed as a linear connection $\nabla$ on $\mathcal{T}\to M$, called the \emph{tractor connection}, which preserves the bundle metric on $\mathcal{T}\to M$. The tractor connection on the standard tractor bundle induces a linear connection on each tractor bundle in the obvious way. Here we will construct $(\mathcal{T},\nabla)$ without reference to $(\mathcal{G},\omega)$.

Following \cite{GoverGraham2005} we take $\cT$ to be the set of equivalence classes of pairs $(\theta, (\sigma,\mu^1,\rho))$, where $\theta$ is a contact form and $(\sigma,\mu^1,\rho)\in \mathcal{E}(0,1)\oplus\mathcal{E}^{1}(-1,0)\oplus\mathcal{E}(-1,0)$, under the equivalence relation: $(\theta, (\sigma,\mu^1,\rho))\sim (\thetah, (\hat{\sigma},\hat{\mu}^1,\hat{\rho}))$ if $\thetah=e^{\Ups}\theta$ and
\begin{equation}\label{TractorTransformation}
\left(\begin{array}{c}
\hat{\sigma}\\
\hat{\mu}^1\\
\hat{\rho}
\end{array}\right)
=
\left(
\begin{array}{ccc}
1 & 0 & 0\\
\Ups^{1} & 1 & 0 \\
-\frac{1}{2}(\Ups^{1}\Ups_{1}-i\Ups_0) & - \Ups_{1} & 1
\end{array}
\right)\left(\begin{array}{c}
\sigma\\
\mu^{1}\\
\rho
\end{array}\right)
\end{equation}
where $\Ups_1=\nabla_1 \Ups$, $\Ups^1=\bh^{1\oneb}\Ups_{\oneb}$ with $\Ups_{\oneb}=\nabla_{\oneb} \Ups$, and $\Ups_0 = \nabla_0 \Ups$. The canonical filtration of $\mathcal{T}$ is immediately evident, fixing a contact form $\theta$ this is given by
\begin{equation*}
\mathcal{T}^1
=
\left\{\left(\begin{array}{c}
0\\
0\\
*
\end{array}\right)\right\}
\subset
\mathcal{T}^0
=
\left\{\left(\begin{array}{c}
0\\
*\\
*
\end{array}\right)\right\}
\subset
\mathcal{T}.
\end{equation*}
If $(\theta, (\sigma,\mu^1,\rho))\sim (\thetah, (\hat{\sigma},\hat{\mu}^1,\hat{\rho}))$ then one easily checks that
\begin{equation*}
 2\hat{\sigma}\hat{\rho} + \hat{\mu}^1\hat{\mu}_1 = 2\sigma\rho + \mu^1\mu_1,
\end{equation*}
which defines by polarization a signature $(2,1)$ Hermitian bundle metric $h$ on $\cT$.
We will adopt the abstract index notation $\cE^A$ for $\mathcal{T}$, and $\cE^{\Ab}$ for $\overline{\mathcal{T}}$, using capitalized Latin letters from the start of the alphabet for our abstract indices. The Hermitian bundle metric $h$ is then written as $h_{A\Bb}$. Decomposing $\cE^A$ with respect to any choice of contact form $\theta$, we have
\begin{equation*}
h_{A\Bb} = \left(
\begin{array}{ccc}
0 & 0 & 1\\
0 & \bh_{1\oneb} & 0 \\
1 & 0 & 0
\end{array}
\right).
\end{equation*}

The line bundle $\cE(-1,0)$ is naturally included in $\cT$ by the map
\begin{equation*}
\rho \mapsto \left(\begin{array}{c}
0\\
0\\
\rho
\end{array}\right).
\end{equation*}
The map $\cE(-1,0)\to\cE^A$ corresponds to a canonical section $\boldsymbol{Z}^A$ of $\cE^A\otimes\cE(1,0)$, known as the \emph{canonical tractor}. The canonical tractor also induces a canonical projection $\cE^A \to \cE(0,1)$ taking $v^A$ to $\sigma = h_{A\bar{B}}v^A\boldsymbol{Z}^{\bar{B}}$. This corresponds to the obvious projection
\begin{equation*}
\left(\begin{array}{c}
\sigma\\
\mu^1\\
\rho
\end{array}\right)
\mapsto \sigma.
\end{equation*}

If $M$ is a strictly pseudoconvex hypersurface in $\mathbb{CP}^2$, then $\cE(-1,0)$ is the restriction of the tautological line bundle $\mathcal{O}(-1)$ to $M$ and $\cT=\cE^A$ can be identified with the restriction of the tautological rank $3$ complex vector bundle over $\mathbb{CP}^2$ (coming from the projection $\mathbb{C}^3\setminus\{0\}\to\mathbb{CP}^2$) to $M$. The canonical tractor can then be identified with the Euler field on $\mathbb{C}^3$, whence the notation $\boldsymbol{Z}$. From this point of view, however, the origins of the tractor metric $h$ and particularly of the tractor connection are more subtle.

\subsubsection*{The tractor connection}
In order to define the tractor connection, we recall the higher order pseudohermitian curvatures $T_1$ and $S$ from \cref{eqn:T_1} and \cref{eqn:S}. With respect to a choice of contact form $\theta$ the tractor connection on a section $v^A \overset{\theta}{=} (\sigma,\mu^1,\rho)$ is then given by
\begin{equation}\label{TractorConnection1}
\nabla_1 v^{A}\overset{\theta}{=}
\left(\begin{array}{c}
\nabla_1\sigma\\
\nabla_1\mu^1 + \rho + \frac{1}{4} R \sigma\\
\nabla_1\rho-iA_{11}\mu^1-\sigma T_1
\end{array}\right),
\end{equation}
\begin{equation}\label{TractorConnection1b}
\nabla_{\oneb}v^{A}\overset{\theta}{=}
\left(\begin{array}{c}
\nabla_{\oneb}\sigma-\mu_{\oneb}\\
\nabla_{\oneb}\mu^1-iA_{\oneb}{^1}\sigma\\
\nabla_{\oneb}\rho-\frac{1}{4} R\mu_{\oneb}+\sigma T_{\oneb}
\end{array}\right),
\end{equation}
and
\begin{equation}\label{TractorConnection0}
\nabla_{0} v^{A}\overset{\theta}{=}
\left(\begin{array}{c}
\nabla_{0}\sigma-\frac{i}{12}R\sigma+i\rho\\
\nabla_{0}\mu^1+\frac{i}{6}R\mu^1-2i\sigma T^1\\
\nabla_{0}\rho-\frac{i}{12}R\rho-2iT_1\mu^1-iS\sigma
\end{array}\right).
\end{equation}
To verify that these formulae give rise to a well defined connection one needs to check that the right hand sides of \cref{TractorConnection1} and \cref{TractorConnection1b} transform according to \cref{TractorTransformation}. For \cref{TractorConnection0} one also needs to take into account the change in the Reeb direction, see \cref{1-FormTransformation}. From these formulae it is easy to see that the tractor connection preserves the tractor metric $h_{A\bar{B}}$.

The tractor curvature $\kappa$ is a $2$-form valued in (trace free skew-Hermitian) endomorphisms of the standard tractor bundle. Given a choice of contact form $\theta$, $\kappa$ may be decomposed into three components $\kappa_{1\oneb}{}_A{}^B$, $\kappa_{10}{}_A{}^B$, and $\kappa_{\oneb 0}{}_A{}^B$, defined by
\begin{align*}
\nabla_1\nabla_{\oneb} v^B -\nabla_{\oneb}\nabla_1 v^B + i \bh_{1\oneb}\nabla_0 v^B & = \kappa_{1\oneb A}{}^B v^A;\\
\nabla_1\nabla_{0} v^B - \nabla_{0}\nabla_1 v^B - A^{\oneb}{_{1}}\nabla_{\oneb}v^B &= \kappa_{10 A}{}^B v^A;\\
\nabla_{\oneb}\nabla_{0} v^B - \nabla_{0}\nabla_{\oneb} v^B - A^{1}{_{\oneb}}\nabla_{1}v^B &= \kappa_{\oneb 0A }{}^B v^A
\end{align*}
for any section $v^A$ of $\cE^A$ (the tractor connection is coupled with the Tanaka-Webster connection of $\theta$ in order to define the iterated covariant derivatives). By definition the component $\kappa_{1\oneb A}{}^B$ of the tractor curvature is a CR invariant, i.e.\ it does not depend on the choice of $\theta$. However, a straightforward calculation shows that $\kappa_{1\oneb}{}_A{}^B=0$. The vanishing of $\kappa_{1\oneb}{}_A{}^B$ implies, by \cref{1-FormTransformation}, that $\kappa_{10 A}{}^B$ and $\kappa_{\oneb 0 A}{}^B$ are CR invariant (this phenomenon is special to $3$-dimensional CR structures). A straightforward calculation using the above formulae for the tractor connection, the formulae \cref{eqn:commuting-11bar-derivatives-on-densities,eqn:commuting-10-derivatives-on-densities} for the curvature of the density line bundles, and the definitions of $T_1$ and $S$, gives
\begin{equation}\label{eqn:tractor-curvature}
\ka_{10 A}{}^B v^A \overset{\theta}{=} \left(\begin{array}{c}
0\\
0\\
\sigma Y_1 + i\mu^1 Q_{11}
\end{array}\right) = \left(
\begin{array}{ccc}
0 & 0 & 0\\
0 & 0 & 0 \\
Y_1 & iQ_{11} & 0
\end{array}
\right)
\left(\begin{array}{c}
\sigma\\
\mu^1\\
\rho
\end{array}\right)
\end{equation}
where $Q_{11}$ is given by \cref{eqn:Cartan-umbilical-tensor}, and
$$Y_1= -i\nabla_1 S + \nabla_0 T_1 + \frac{i}{2}RT_1-3A_{11}T^1.
$$
The CR invariance of $Q_{11}$ then follows immediately from the CR invariance of $\kappa_{10 A}{}^B$ and the transformation law \cref{TractorTransformation}. On the other hand, $Y_1$ is not CR invariant, rather the transformation law \cref{TractorTransformation} implies that if $\thetah = e^{\Ups}\theta$ then $\hat{Y}_1 = Y_1 - Q_{11}\Upsilon^1$. (The pair $Q_{11}$ and $Y_1$, respectively, are highly analogous to the Weyl curvature and Cotton tensor in $4$-dimensional conformal geometry.) Since the tractor connection preserves the tractor metric we have $\kappa_{\oneb 0 A}{}^B = - h_{A\bar{D}} h^{B\bar{C}} \overline{\kappa_{10 C}{}^D}$, giving
\begin{equation}\label{eqn:tractor-curvature-oneb0}
\ka_{\oneb 0 A}{}^B v^A \overset{\theta}{=} \left(
\begin{array}{ccc}
0 & 0 & 0\\
iQ_{\oneb}{}^1 & 0 & 0 \\
-Y_{\oneb} & 0 & 0
\end{array}
\right)
\left(\begin{array}{c}
\sigma\\
\mu^1\\
\rho
\end{array}\right)
\end{equation}
where $Y_{\oneb} = \overline{Y_1}$.

\subsection{The adjoint tractor bundle and the obstruction as a divergence}
Let $\mathbb{C}^{2,1}$ denote $\mathbb{C}^3$ equipped with the signature $(2,1)$ Hermitian inner product
\begin{equation*}
\langle (z_0,z_1,z_2), (w_0,w_1,w_2)\rangle = z_0\overline{w_2} + z_1\overline{w_1} + z_2\overline{w_0}
\end{equation*}
chosen so that the standard first and last basis vectors are isotropic. Let $G=\mathrm{SU}(2,1)$ be the linear group preserving the inner product, with Lie algebra
\begin{equation*}
\mathfrak{su}(2,1)=\left\{
\left(
\begin{array}{ccc}
a & z & i\phi\\
w & -2i\,\mathrm{Im}\, a  & -\bar{z} \\
i\psi & -\bar{w} & -\bar{a}
\end{array}
\right) : \phi,\psi\in \mathbb{R},\; a,z,w\in \mathbb{C} \right\}.
\end{equation*}

The adjoint tractor bundle is the bundle induced from $\cG$ by the adjoint representation of $G$ on its Lie algebra $\mathbb{V}=\fg$. Since $\fg$ consists of the trace-free skew-Hermitian endomorphisms of $\mathbb{C}^{2,1}$, the adjoint tractor bundle $\cA\to M$ is the subbundle of $\mathrm{End}(\cT)$ consisting of trace-free skew-Hermitian endomorphisms of $\cT$. A section $s\in \Gamma(\cA)$ may be written with respect to a choice of contact form $\theta$ as
\begin{equation*}
 s_A{}^B \overset{\theta}{=}\left(\begin{array}{ccc}
\mu & \upsilon_1 & iu\\
\nu^1 & -2i \mathrm{Im}\mu& -\upsilon^1 \\
i\lambda & -\nu_1 & -\overline{\mu}
\end{array}\right).
\end{equation*}
If $\thetah=e^{\Ups}\theta$ then
\begin{equation*}
 s_A{}^B \overset{\thetah}{=}
\left(\!
\begin{array}{ccc}
1 & 0 & 0\\
\Ups^{1} & 1 & 0 \\
-\frac{1}{2}(\Ups^{1}\Ups_{1}-i\Ups_0) & - \Ups_{1} & 1
\end{array}
\right) \!
 \left(\begin{array}{ccc}
\mu & \upsilon_1 & iu\\
\nu^1 & -2i \mathrm{Im}\mu& -\upsilon^1 \\
i\lambda & -\nu_1 & -\overline{\mu}
\end{array}\right)\!
\left(\!
\begin{array}{ccc}
1 & 0 & 0\\
-\Ups^{1} & 1 & 0 \\
-\frac{1}{2}(\Ups^{1}\Ups_{1}+i\Ups_0) & \Ups_{1} & 1
\end{array}
\right).
\end{equation*}

If a section of $\cA$ is given by
\begin{equation*}
 s_A{}^B \overset{\theta}{=}\left(\begin{array}{ccc}
\mu & \upsilon_1 & iu\\
\nu^1 & *& * \\
i\lambda & * & *
\end{array}\right)
\end{equation*}
then
\begin{equation}\label{eqn:adjoint-tractor-connection-1}
 \nabla_1 s_A{}^B \overset{\theta}{=}\left(\begin{array}{ccc}
\nabla_1\mu - \frac{1}{4}R\ups_1 + iuT_1 & \nabla_1\ups_1 - A_{11}u & i\nabla_1 u - \ups_1\\
\nabla_1\nu^1 + i\lambda + \frac{1}{4}R(2\mu-\overline{\mu})-\ups^1T_1 & * & * \\
i\nabla_1\lambda - iA_{11}\nu^1 +\frac{1}{4}R\nu_1 -(\mu+\overline{\mu})T_1 & * & *
\end{array}\right)
\end{equation}
\begin{equation}\label{eqn:adjoint-tractor-connection-oneb}
 \nabla_{\oneb} s_A{}^B \overset{\theta}{=}\left(\begin{array}{ccc}
\nabla_{\oneb}\mu - \nu_{\oneb} + iA_{\oneb}{}^1\ups_1-iuT_{\oneb} & \nabla_{\oneb}\ups_1 + (2\mu-\overline{\mu})\bh_{1\oneb}+\frac{i}{4}uR\bh_{1\oneb} & i\nabla_{\oneb} u + \ups_{\oneb}\\
\nabla_{\oneb}\nu^1 - i(2\mu-\overline{\mu})A_{\oneb}{}^1 +\ups^1T_{\oneb}& * & * \\
i\nabla_{\oneb}\lambda -\frac{1}{4}R\nu_{\oneb} -iA_{\oneb}{}^1\nu_1 +(\mu+\overline{\mu})T_{\oneb}& * & *
\end{array}\right)
\end{equation}
\begin{equation}
 \nabla_0 s_A{}^B \overset{\theta}{=}\!\left(\!\!\begin{array}{ccc}
\nabla_0\mu -\lambda + 2i \ups_1T^1 - Su & \nabla_0\ups_1 \!-\! \frac{i}{4}R\ups_1 \!-\! i\nu_1 \!-\! 2uT_1 & i\nabla_0 u \!-\! i(\mu\!+\!\overline{\mu})\\
\nabla_0\nu^1 +\frac{i}{4}R\nu^1 - 2i(2\mu-\overline{\mu}) T^1 -iS\ups^1& * & * \\
i\nabla_0\lambda -2i(T_1\nu^1+ T^1\nu_1) - iS(\mu+\overline{\mu})  & * & *
\end{array}\!\!\right)
\end{equation}

The tractor curvature $\kappa$ satisfies the Bianchi identity, $\dee^{\nabla}\kappa=0$, which can be written in terms of the components as $\nabla_{1}\kappa_{\oneb 0A}{}^B - \nabla_{\oneb} \kappa_{10A}{}^B = 0$ (i.e.\  $\nabla^1\kappa_{10A}{}^B = \nabla^{\oneb}\kappa_{\oneb 0A}{}^B$).
\begin{lemma}\label{lem:obstruction-flat-as-divergence}
Let $(M,H,J)$ be a CR $3$-manifold. Then the CR obstruction density $\mathcal{O}$ vanishes if and only if $\nabla^1\kappa_{10A}{}^B=0$ (equivalently $\nabla^{\oneb}\kappa_{\oneb 0A}{}^B=0$).
\end{lemma}
\begin{proof}
Fix a background contact form $\theta$. By a straightforward calculation using \cref{eqn:tractor-curvature}, \cref{eqn:tractor-curvature-oneb0}, \cref{eqn:adjoint-tractor-connection-1} and \cref{eqn:adjoint-tractor-connection-oneb}, from the Bianchi identity $\nabla_{1}\kappa_{\oneb 0A}{}^B - \nabla_{\oneb} \kappa_{10A}{}^B = 0$ one obtains that $Y_1 = -i\nabla^1 Q_{11}$. By the same calculation, using this identity, one obtains that 
\begin{equation*}
\nabla^1\kappa_{10A}{}^B \overset{\theta}{=} 
\left(\begin{array}{ccc}
0 & 0 & 0\\
0 & 0 & 0 \\
-i(\nabla^1\nabla^1Q_{11} - iA^{11}Q_{11}) & 0 & 0
\end{array}\right).
\end{equation*}
The lemma follows immediately by \cref{eqn:obstruction-density}.
\end{proof}

\subsection{Proof of \texorpdfstring{\cref{thm:Symmetries}}{Theorem 1.3}}

The proof of \cref{thm:Symmetries} makes use of a universal prolongation formula for the infinitesimal automorphism equation in a parabolic geometry, which puts infinitesimal automorphisms in one to one correspondence with nontrivial sections of the adjoint tractor bundle satisfying a first order (prolonged) equation \cite{Cap2008}. In the statement of this result for the $3$-dimensional CR case (\cref{lem:Cap08}) we refer to the operator $L$ given by the following proposition, an example of a so called \emph{BGG-splitting operator}.
\begin{proposition}[\cite{Cap2008,CE2018-symmetries-and-deformations}]
The BGG-splitting operator $L:\cE(1,1)\to\cA$ is given by
\begin{equation*}
L u \overset{\theta}{=}\left(\begin{array}{ccc}
\mu & \upsilon_1 & iu\\
\nu^1 & *& * \\
i\lambda & * & *
\end{array}\right)
\end{equation*}
where $\ups_1=i\nabla_1 u$, $\mu=\frac{1}{3}(\nabla_0 u - \nabla^1\ups_1 -\frac{i}{4}uR )$, $\nu^1=\frac{1}{3}(i\nabla_0\ups^1 +2\nabla^1\mu-\nabla^1\overline{\mu}+2iA^{11}\ups_1-3iuT^1)$, and $\lambda=\frac{1}{4i}(2i\nabla_0\mathrm{Re}\mu \,\!+\! \nabla^1\nu_1\!-\!\nabla_1\nu^1 \!-\!\frac{3i}{2}R\,\mathrm{Im}\mu \!+\! 3(\ups^1T_1\!-\!\ups_1T^1) \!-\! 2iSu)$.
\end{proposition}
The following lemma, which is part (1) of the proposition of section 3.2 of \cite{Cap2008}, gives the prolonged system corresponding to the CR infinitesimal automorphism equation.
\begin{lemma}\label{lem:Cap08}
Let $(M,H,J)$ be a CR $3$-manifold. The vector field $X$ is an infinitesimal CR symmetry if and only if $\nabla s = -X\,\hook\, \kappa$, where $s=L u$ with $u=\btheta(X)$.
\end{lemma}

\begin{proof}[Proof of \cref{thm:Symmetries}]
Suppose $\mathcal{O}=0$. Fix a contact form $\theta$ and let $\nabla$ denote the Tanaka-Webster connection of $\theta$ coupled with the CR tractor connection. By \cref{lem:obstruction-flat-as-divergence}, since $\mathcal{O}=0$ we have $\nabla^{\oneb}\kappa_{\oneb0A}{}^B=0$. If $X$ is an infinitesimal CR symmetry then by \cref{lem:Cap08}, since $\kappa_{1\oneb A}{}^B=0$, we have $\nabla_1 s_A{}^B = u\kappa_{10A}{}^B$, where $u = \btheta(X)$. Note that $s_A{}^{C}s_B{}^A\,\nabla^{\oneb}\kappa_{\oneb 0C}{}^B = s_A{}^C \,(\nabla^{\oneb}\kappa_{\oneb 0B}{}^A) \,s_C{}^B$ is a density of weight $(-2,-2)$ so can be invariantly integrated. Integrating by parts we obtain
\begin{align}\label{eqn:int-by-parts}
0=\int s_A{}^C \,(\nabla^{\oneb}\kappa_{\oneb 0B}{}^A) \,s_C{}^B & = - \int \left( (\nabla^{\oneb}s_A{}^C)\,\kappa_{\oneb 0B}{}^A \,s_C{}^B + s_A{}^C \,\kappa_{\oneb 0B}{}^A \,\nabla^{\oneb} s_C{}^B \right)\\
\nonumber &= - \int u h^{1\oneb} \left( \kappa_{10A}{}^C \,\kappa_{\oneb 0B}{}^A\, \,s_C{}^B + s_A{}^C \,\kappa_{\oneb 0B}{}^A \,\kappa_{10C}{}^B \right).
\end{align}
Now by \cref{eqn:tractor-curvature} and \cref{eqn:tractor-curvature-oneb0} we have
\begin{equation*}
h^{1\oneb}\kappa_{\oneb 0B}{}^A \,\kappa_{10C}{}^B  \overset{\theta}{=}
\left(
\begin{array}{ccc}
0 & 0 & 0\\
iQ_{\oneb}{}^1 & 0 & 0 \\
-Y_{\oneb} & 0 & 0
\end{array}
\right)
\left(
\begin{array}{ccc}
0 & 0 & 0\\
0 & 0 & 0 \\
Y^{\oneb} & iQ_{1}{}^{\oneb} & 0
\end{array}
\right)
=\left(
\begin{array}{ccc}
0 & 0 & 0\\
0 & 0 & 0 \\
0 & 0 & 0
\end{array}
\right)
\end{equation*}
and 
\begin{equation*}
h^{1\oneb}\kappa_{10A}{}^C \,\kappa_{\oneb 0B}{}^A  \overset{\theta}{=}
\left(
\begin{array}{ccc}
0 & 0 & 0\\
0 & 0 & 0 \\
Y^{\oneb} & iQ_{1}{}^{\oneb} & 0
\end{array}
\right)
\left(
\begin{array}{ccc}
0 & 0 & 0\\
iQ_{\oneb}{}^1 & 0 & 0 \\
-Y_{\oneb} & 0 & 0
\end{array}
\right)
=\left(
\begin{array}{ccc}
0 & 0 & 0\\
0 & 0 & 0 \\
- |Q|^2 & 0 & 0
\end{array}
\right)
\end{equation*}
where $|Q|^2= Q_{11}Q^{\oneb\oneb}$. Hence \cref{eqn:int-by-parts} simplifies to
\begin{equation*}
\int u h^{1\oneb}\kappa_{10A}{}^C \,\kappa_{\oneb 0B}{}^A\, \,s_C{}^B = 0
\end{equation*}
and we have
\begin{align*}
h^{1\oneb}\kappa_{10A}{}^C \,\kappa_{\oneb 0B}{}^A\, \,s_C{}^B  \,&\overset{\theta}{=}\,
\mathrm{tr} \, \left(
\begin{array}{ccc}
0 & 0 & 0\\
0 & 0 & 0 \\
- |Q|^2 & 0 & 0
\end{array}
\right) \left(\begin{array}{ccc}
\mu & \upsilon_1 & iu\\
\nu^1 & -2i \mathrm{Im}\mu& -\upsilon^1 \\
i\lambda & -\nu_1 & -\overline{\mu}
\end{array}\right) \\
&=\, \mathrm{tr} \left(
\begin{array}{ccc}
0 & 0 & 0\\
0 & 0 & 0 \\
-\mu |Q|^2  & -\ups_1|Q|^2 & -iu|Q|^2
\end{array}
\right) 
= -iu|Q|^2.
\end{align*}
We conclude that $\int u^2|Q|^2 =0$. Since $u$ cannot vanish on an open set, the result follows.
\end{proof}

\bibliographystyle{abbrv}

\newcommand{\noopsort}[1]{}

\end{document}